\documentclass[12pt]{amsart}
\usepackage{graphicx}
\usepackage{hyperref}
\usepackage{setspace}
\usepackage{enumerate}
\usepackage{amsmath,amsfonts,amssymb,amscd}

\newtheorem{thm}{Theorem}[section]

\newtheorem{lem}[thm]{Lemma}
\newtheorem{claim}[thm]{Claim}
\newtheorem{prop}[thm]{Proposition}

\theoremstyle{definition}

\theoremstyle{remark}
\newtheorem{rem}[thm]{Remark}
\numberwithin{equation}{section}

\begin{document}

\title[On an oscillatory integral involving a homogeneous form]{On an oscillatory integral \\ involving a homogeneous form}

\author{Shuntaro Yamagishi}
\address{School of Mathematics, University of Bristol, Bristol, BS8 1TW, United Kingdom}
\email{sy17629@bristol.ac.uk}
\indent

\date{Revised on \today}

\begin{abstract}
Let $F \in \mathbb{R}[x_1, \ldots, x_n]$ be a homogeneous form of degree $d > 1$ satisfying $(n - \dim V_{F}^*) > 4$,
where $V_F^*$ is the singular locus of $V(F) = \{ \mathbf{z} \in {\mathbb{C}}^n: F(\mathbf{z}) = 0 \}$.
Suppose there exists $\mathbf{x}_0 \in (0,1)^n \cap (V(F) \backslash V_F^*)$.
Let $\mathbf{t} = (t_1, \ldots, t_n) \in \mathbb{R}^n$. Then for a smooth function $\varpi:\mathbb{R}^n \rightarrow \mathbb{R}$ with its support contained in a small neighbourhood of $\mathbf{x}_0$,
we prove
$$
\Big{|} \int_{0}^{\infty} \cdots \int_{0}^{\infty}
\varpi(\mathbf{x}) x_1^{i t_1} \cdots x_n^{i t_n} e^{2 \pi i \tau F(\mathbf{x})} d \mathbf{x}  \Big{|} \ll \min \{ 1, |\tau|^{-1} \},
$$
where the implicit constant is independent of $\tau$ and $\mathbf{t}$.
\end{abstract}


\maketitle

\section{Introduction}
\label{secintro}
Let $F \in \mathbb{R}[x_1, \ldots, x_n]$ be a homogeneous form of degree $d$ and
let $V(F; \mathbb{R}) = \{ \mathbf{z} \in \mathbb{R}^n : F(\mathbf{z}) = 0 \}$.
We let $V^*_{F}$ denote the singular locus of $V(F) = \{ \mathbf{z} \in {\mathbb{C}}^n: F(\mathbf{z}) = 0 \}$ which is an affine variety (not necessarily irreducible) in $\mathbb{A}^n_{\mathbb{C}}$ defined by
\begin{equation}
\label{sing loc}
V_{F}^* = \left\{ \mathbf{z} \in \mathbb{C}^n:  \nabla F(\mathbf{z}) = \mathbf{0} \right\},
\end{equation}
where
$\nabla F = \left( \frac{\partial F}{\partial x_1}, \ldots,  \frac{\partial F}{\partial x_n} \right)$.
The main purpose of this paper is to prove the following.
\begin{thm}
\label{mainthm}
Let $F \in \mathbb{R}[x_1, \ldots, x_n]$ be a homogeneous form of degree $d > 1$ satisfying $(n - \dim V_{F}^*) > 4$.
Let
$\mathbf{r} = (r_1, \ldots, r_n) \in [\theta_1, \theta_1'] \times \cdots \times [\theta_n, \theta_n']$, where $\theta_j \leq \theta'_j$ $(1 \leq j \leq n)$,
and  $\mathbf{t} = (t_1, \ldots, t_n) \in \mathbb{R}^n$.
Suppose $\mathbf{x}_0 \in (0,1)^n \cap V(F; \mathbb{R})$ is non-singular.
Let $\varpi: \mathbb{R}^n \rightarrow \mathbb{R}$ be a smooth function whose support is contained in $(\mathbf{x}_0 + [- \delta_0, \delta_0]^n)$ for $\delta_0 > 0$.
Then provided $\delta_0$ is sufficiently small, we have
\begin{eqnarray}
\label{mainthmint}
\Big{|} \int_{0}^{\infty} \cdots \int_{0}^{\infty}
\varpi(\mathbf{x}) x_1^{r_1 + i t_1} \cdots x_n^{r_n + i t_n} e^{2 \pi i \tau F(\mathbf{x})} d \mathbf{x}  \Big{|} \ll \min \{ 1, |\tau|^{-1} \},
\end{eqnarray}
where the implicit constant is independent of $\tau$, $\mathbf{r}$ and $\mathbf{t}$.
\end{thm}
In the statement of the theorem, by $\mathbf{x}_0 \in V(F; \mathbb{R})$ is non-singular we mean $\mathbf{x}_0 \not \in V_F^*$, i.e.
there exists  $1 \leq j_0 \leq n$ such that $\partial F / \partial x_{j_0} (\mathbf{x}_0) \not = 0.$
Also we note the implicit constant in (\ref{mainthmint}) is independent of $\mathbf{r}$ but it will depend on $\theta_j$ and $\theta'_j$
$(1 \leq j \leq n)$. The key feature of the result is that the bound is uniform in $\mathbf{t}$; the result can be deduced easily
for a fixed $\mathbf{t} \in \mathbb{R}^n$ (for example, by \cite[Lemma 10]{HB}), but obtaining the uniformity is quite delicate and this is where the challenge lies. We make use of an explicit version of the inverse function theorem, the stationary phase method,
basic oscillatory integral estimates, some differential geometry and algebraic geometry over $\mathbb{R}$  to achieve this.

The oscillatory integral in (\ref{mainthmint}) is related to the singular integral which appears in an application of the Hardy-Littlewood circle method
involving a homogeneous form $F \in \mathbb{Z}[x_1, \ldots, x_n]$. In fact, Theorem \ref{mainthm} is one of the important ingredients in the
forthcoming paper by the author. Since similar integrals come up often in analytic number theory, for example see \cite[Lemma 4.9]{V},
Theorem \ref{mainthm} and other estimates in this paper may find further useful applications elsewhere. 

\textit{Acknowledgments.} The author would like to thank Matthew Beckett, Tim Browning, Vinay Kumaraswamy and Ian Petrow for many helpful discussions.
\section{Preliminaries}

Let $\| \cdot \|$ denote the $L^2$-norm on $\mathbb{R}^n$. Given $\mathcal{X} \subseteq \mathbb{R}^n$,
let $\overline{\mathcal{X}}$ denote the closure of $\mathcal{X}$,
$\textnormal{int}(\mathcal{X}) = \mathbb{R}^n \backslash \overline{(\mathbb{R}^n \backslash \mathcal{X})}$
the interior of $\mathcal{X}$ and $\partial \mathcal{X} = \overline{\mathcal{X}} \backslash \textnormal{int}(\mathcal{X})$ the boundary of
$\mathcal{X}$. We present a proof of the following explicit version of the inverse function theorem in Appendix \ref{AppB}.
Given a function $\mathfrak{F} = (\mathfrak{F}_1, \ldots, \mathfrak{F}_n): \mathbb{R}^n \rightarrow \mathbb{R}^n$
differentiable at $\mathbf{x}_0$, we denote by $\textnormal{Jac} \mathfrak{F} (\mathbf{x}_0)$
the $n \times n$ Jacobian matrix of $\mathfrak{F}$ at $\mathbf{x}_0$, i.e. $[\partial \mathfrak{F}_i / \partial x_j (\mathbf{x}_0) ]$.

\begin{thm}[Explicit inverse function theorem]
\label{thm exp inv}
Let $\mathfrak{F} = (\mathfrak{F}_1, \ldots, \mathfrak{F}_n): \mathbb{R}^n \rightarrow \mathbb{R}^n$ be smooth.
Let $\mathbf{x}_0 \in \mathbb{R}^n$ and suppose $A = \textnormal{Jac} \mathfrak{F} (\mathbf{x}_0)$ is invertible.
Let $a_{\max}$ be the maximum of the absolute values of the entries of $A$. Let $0 < M < |\det A|/ (n \cdot n! \cdot a_{\max}^{n-1})$.
Let $W \subseteq \mathbb{R}^n$ be a bounded convex open set such that

\textnormal{i)} $\mathbf{x}_0 \in W$,

\textnormal{ii)} $\textnormal{det}\left( \textnormal{Jac}\mathfrak{F}(\mathbf{x}) \right) \not = 0$ $(\mathbf{x} \in W)$, and

\textnormal{iii)}
$$
\Big{|} \frac{\partial \mathfrak{F}_i}{ \partial x_j} (\mathbf{x}) - \frac{\partial \mathfrak{F}_i}{ \partial x_j} (\mathbf{x}_0) \Big{|} < M \ \ \ (\mathbf{x} \in W, \ 1 \leq i, j \leq n).
$$
Let
$$
V = \{ \mathbf{y} \in \mathbb{R}^n : \| \mathbf{y} - \mathfrak{F}(\mathbf{x}_0) \| < m/2  \},
$$
where
$$
m = \min_{\mathbf{x} \in \partial W} \| \mathfrak{F}(\mathbf{x}) - \mathfrak{F}(\mathbf{x}_0) \|.
$$
Then $\mathfrak{F}^{-1}$ is well-defined and smooth on $V$, and
$\mathfrak{F}^{-1}(V) \subseteq W$ is diffeomorphic to $V$.
\end{thm}

Let $F \in \mathbb{R}[x_1, \ldots, x_n]$, not necessarily homogenous.
Let us define
$$
\mathfrak{G}_{1,2} (\mathbf{x}) = \left( x_1 \frac{\partial^{2} F}{\partial x_1^2 } (\mathbf{x})  + \frac{\partial F}{\partial x_1} (\mathbf{x})   \right) \cdot \left( x_2 \frac{\partial^{2} F}{\partial x_2^2}(\mathbf{x})
+ \frac{\partial F}{\partial x_2}(\mathbf{x})  \right) -
x_1 x_2 \left( \frac{\partial^{2} F}{\partial{x_1} \partial{x_2}}(\mathbf{x}) \right)^2 .
$$

We consider the following two cases separately.

Case (I): There exist $\mathbf{x}_0 \in (0,1)^n$ and $\delta > 0$ such that
given any $\mathbf{x} \in (\mathbf{x}_0 + [- \delta, \delta]^n)$ we have
\begin{eqnarray}
\label{set of condn}
\mathfrak{G}_{1,2} (\mathbf{x}) \not = 0,  \  \  \  \frac{\partial^{2} F}{\partial x_1 \partial x_2 } (\mathbf{x})  \not = 0,
\end{eqnarray}
$$
x_1 \frac{\partial^{2} F}{\partial x_1^2 }(\mathbf{x})  + \frac{\partial F}{\partial x_1}(\mathbf{x}) \not = 0 \  \   \text{  and  }   \  \    x_2 \frac{\partial^{2} F}{\partial x_2^2 }(\mathbf{x})  + \frac{\partial F}{\partial x_2}(\mathbf{x}) \not = 0.
$$

Case (II): We have $\frac{\partial^2 F}{\partial x_1 \partial x_2} \equiv 0$. Furthermore, there exist $\mathbf{x}_0 \in (0,1)^n$ and $\delta > 0$ such that
given any $\mathbf{x} \in (\mathbf{x}_0 + [- \delta, \delta]^n)$ we have
\begin{eqnarray}
\label{set of condn'}
\mathfrak{G}_{1,2} (\mathbf{x}) \not = 0,
\ \
x_1 \frac{\partial^{2} F}{\partial x_1^2 }(\mathbf{x})  + \frac{\partial F}{\partial x_1}(\mathbf{x}) \not = 0 \  \   \text{  and  }   \  \
x_{2} \frac{\partial^{2} F}{\partial x_2^2 } (\mathbf{x})  + \frac{\partial F}{\partial x_2} (\mathbf{x}) \not = 0.
\end{eqnarray}

We prove the following result from which we deduce Theorem \ref{mainthm} in Section \ref{secdeducemain}.
\begin{prop}
\label{mainprop}
Let $F  \in \mathbb{R}[x_1, \ldots, x_n]$ be a polynomial, not necessarily homogeneous,
satisfying the hypotheses of either Case \textnormal{(I)} or Case \textnormal{(II)}.
Let $\varpi : \mathbb{R}^n \rightarrow \mathbb{R}$ be a smooth function satisfying
\begin{eqnarray}
\label{smoothweightbound}
\max_{\mathbf{x} \in \mathbb{R}^n } | \varpi (\mathbf{x}) | +
\max_{1 \leq j \leq n}  \max_{\mathbf{x} \in \mathbb{R}^n }  \Big{|} \frac{\partial \varpi}{ \partial x_j }  (\mathbf{x}) \Big{|} +
\max_{1 \leq j \leq k \leq n}  \max_{\mathbf{x} \in \mathbb{R}^n }  \Big{|} \frac{\partial^2 \varpi}{ \partial x_j \partial x_k} (\mathbf{x}) \Big{|} < \mathfrak{C},
\end{eqnarray}
and its support contained in $(\mathbf{x}_0 + [- \delta_0, \delta_0]^n)$ for $\delta_0 > 0$.
Suppose $\delta_0$ is sufficiently small \textnormal{(}the choice of $\delta_0$ here depends only on $\delta$, $F$ and $\mathbf{x}_0$\textnormal{)}. Then we have
\begin{eqnarray}\label{mainintegral2}
\Big{|} \int_{- \infty}^{\infty} \cdots \int_{- \infty}^{\infty}
\varpi(\mathbf{x}) x_1^{i t_1} \cdots x_n^{i t_n} e^{2 \pi i \tau F(\mathbf{x})} d \mathbf{x}  \Big{|} \ll \min \{ 1, |\tau|^{-1} \},
\end{eqnarray}
where the implicit constant is independent of $\tau$, $\mathbf{t}$ and the specific choice of $\varpi$
\textnormal{(}but it will depend on $\mathfrak{C}$, $F$, $\mathbf{x}_0$ and $\delta_0$\textnormal{)}.
\end{prop}

The idea of the proof is the following. In order to achieve the above estimate we only need
to focus on two of the variables. If we can obtain a lower bound for one of the first partial derivatives
then we are done by a basic result on oscillatory integrals. On the other hand, if both of the first partial derivatives are
close to $0$ then we essentially use the stationary phase method. The main complication is making sure all the
estimates we obtain are uniform in our parameters. We begin with the box given in
Case (I) and we shrink it essentially three times so that all the points satisfy certain desired conditions,
and then we obtain the estimate on the integral; the proof for Case (II) is very similar to that of Case (I) and we keep the details to a minimum.
Since the estimate (\ref{mainintegral2}) is trivial when $|\tau| \leq 1$, we assume $|\tau| > 1$ throughout the proof.
Also given $\mathcal{M}, \mathcal{N} \subseteq \mathbb{R}^n$ we use the notations
$\mathcal{M} + \mathcal{N} = \{ \mathbf{m} + \mathbf{n} : \mathbf{m} \in \mathcal{M}, \mathbf{n} \in \mathcal{N}\},$
$\mathcal{M} - \mathcal{N} = \{ \mathbf{m} - \mathbf{n} : \mathbf{m} \in \mathcal{M}, \mathbf{n} \in \mathcal{N}\}$
and $-\mathcal{M} = \{ - \mathbf{m}: \mathbf{m} \in \mathcal{M} \}$.
\section{First Box - Case (I)}
\label{secfirstbox}
Let us suppose $F$ satisfies the hypotheses of Case (I).
Let $\mathcal{B} = (\mathbf{x}_0 + [- \delta, \delta]^n)$,
without loss of generality we assume $\mathcal{B} \subseteq (0,1)^n$, and we define
\begin{eqnarray}
\label{defnm0m1}
m_0 = \min_{ \mathbf{x} \in  \mathcal{B}}  \   | \mathfrak{G}_{1,2} (\mathbf{x}) |,
\end{eqnarray}
\begin{eqnarray}
\label{defm2}
m_1 = \min_{1 \leq i \leq 2} \ \min_{ \mathbf{x} \in  \mathcal{B}} \  \Big{|} x_i \frac{\partial^{2} F}{\partial x_i^2 }(\mathbf{x})  + \frac{\partial F}{\partial x_i} (\mathbf{x}) \Big{|} \ \  \text{ and } \  \
m_2 =  \min_{ \mathbf{x} \in  \mathcal{B}} \  \Big{|} \frac{\partial^{2} F}{\partial x_1 \partial x_2 } (\mathbf{x}) \Big{|}.
\end{eqnarray}
Because we are considering Case (I) it follows that $m_0, m_1, m_2 > 0$.
We denote
\begin{eqnarray}
\label{defnrho1}
\mathcal{B} = [\rho_{1,1}, \rho_{1,2}] \times \cdots \times [\rho_{n,1}, \rho_{n,2}],
\end{eqnarray}
\begin{eqnarray}
\label{defnrho2}
\rho_{\min} =  \min_{1 \leq i \leq n } \rho_{i,1} \ \ \ \text{  and  } \ \  \  \rho_{\max} =  \max_{1 \leq i \leq n} \rho_{i,2}.
\end{eqnarray}
Furthermore, let us denote
\begin{eqnarray}
\label{defnrho3}
\mathcal{B}_0 = [\rho_{1,1}, \rho_{1,2}] \times [\rho_{2,1}, \rho_{2,2}] \ \ \text{  and  } \  \
\mathcal{B}_0' = [\rho_{3,1}, \rho_{3,2}] \times \cdots \times [\rho_{n,1}, \rho_{n,2}].
\end{eqnarray}

Let $\mathbf{u} = (u_1, u_2) = (x_1, x_2)$ and $\mathbf{v} = (x_3, \ldots, x_n)$, and also let $\mathbf{u}_0 = (x_{0,1}, x_{0,2})$ and
$\mathbf{v}_0 = (x_{0,3}, \ldots, x_{0,n})$ so that $\mathbf{x}_0 = (\mathbf{u}_0, \mathbf{v}_0)$. Note if $n=2$ then there is no need to consider the vectors $\mathbf{v}$ and $\mathbf{v}_0$. Let us define $\Psi_{\mathbf{v}} = (\Psi_{\mathbf{v},1}, \Psi_{\mathbf{v},2}): \mathbb{R}^2 \rightarrow \mathbb{R}^2$, where
$$
\Psi_{\mathbf{v},1} (\mathbf{u}) = u_1 \frac{\partial F}{\partial u_1} (u_1, u_2, \mathbf{v})
\ \
\text{  and  }
\ \
\Psi_{\mathbf{v}, 2} (\mathbf{u}) = u_2 \frac{\partial F}{\partial u_2} (u_1, u_2, \mathbf{v}).
$$
Note
\begin{eqnarray}
\label{detJacpsi}
\det( \textnormal{Jac}\Psi_{\mathbf{v}}(\mathbf{u})) = \mathfrak{G}_{1,2}(\mathbf{u}, \mathbf{v}).
\end{eqnarray}

\begin{claim}\label{claim1} Let $\delta_1 > 0$, $\mathcal{B}_1 = (\mathbf{u}_0 + (- \delta_1, \delta_1)^2)$ and  $\mathcal{B}'_1 = (\mathbf{v}_0 + [- \delta_1, \delta_1]^{n-2})$.  Then for $\delta_1 > 0$ sufficiently small, $\Psi$ is a diffeomorphism on $\mathcal{B}_1 \subseteq \mathcal{B}_0$ for any $\mathbf{v} \in \mathcal{B}'_1$.
\end{claim}
\begin{proof}
Given a matrix $\mathcal{L}$ let $[\mathcal{L}]_{i,j}$ denote the $(i,j)$-th entry of $\mathcal{L}$.
First we let
$$
\mathfrak{a} =  \max_{1 \leq i, j \leq 2}  \ \max_{\mathbf{v} \in \mathcal{B}_0'} |  [\textnormal{Jac}\Psi_{\mathbf{v}}(\mathbf{u}_0)]_{i,j}  |
$$
and
$$
M = \frac{m_0}{8 \mathfrak{a} }.
$$
Then for any $\mathbf{v} \in \mathcal{B}'_0$ it follows from (\ref{detJacpsi}) that
$$
M <  \frac{ | \det (\textnormal{Jac}\Psi_{\mathbf{v}}(\mathbf{u}_0) )| }{ 2 \cdot 2! \cdot \max_{1 \leq i, j \leq 2} |  [\textnormal{Jac}\Psi_{\mathbf{v}}(\mathbf{u}_0)]_{i,j}  |  }.
$$
Let $W = (\mathbf{u}_0 + (- \delta', \delta')^2)$ where $\delta' > 0$ is sufficiently small (in particular $\delta' < \delta$).
We verify that $W$ satisfies the three properties of Theorem \ref{thm exp inv}.

i) It is clear from the definition that $\mathbf{u}_0 \in W$.

ii) We have $\det( \textnormal{Jac}\Psi_{\mathbf{v}}(\mathbf{u})) \not = 0$ $(\mathbf{u} \in W)$, because
of (\ref{detJacpsi}), (\ref{set of condn}) and $0 < \delta' < \delta$.

iii) We now show for any $\mathbf{v} \in (\mathbf{v}_0 + (- \delta', \delta')^{n-2})$,
\begin{eqnarray}
\label{ineq ofiii1}
\Big{|} \frac{\partial \Psi_{\mathbf{v}, i}}{\partial u_j}(\mathbf{u})  -  \frac{\partial \Psi_{\mathbf{v}, i} }{\partial u_j}(\mathbf{u}_0)  \Big{|} < M \ \ \
(\mathbf{u} \in W, 1 \leq i, j \leq 2).
\end{eqnarray}
Given any $\mathbf{u} \in W$ it follows from the mean value theorem that there exists $\mathbf{c} \in W$ such that
\begin{eqnarray}
\frac{\partial \Psi_{\mathbf{v},i}}{\partial u_j}(\mathbf{u})  -  \frac{\partial \Psi_{\mathbf{v}, i}}{\partial u_j}(\mathbf{u}_0) = \nabla \frac{\partial \Psi_{\mathbf{v}, i} }{\partial u_j} (\mathbf{c}) \cdot ( \mathbf{u} - \mathbf{u}_0) .
\end{eqnarray}
It can be verified that there exists $C'>0$ depending only on $F$, $\mathcal{B}_0$ and $\mathcal{B}'_0$ such that
$$
\max_{1 \leq i, j, k \leq 2} \max_{ \substack{ \mathbf{u} \in \mathcal{B}_0 \\  \mathbf{v} \in \mathcal{B}'_0  } }
\Big{|} \frac{\partial^2 \Psi_{\mathbf{v},i} }{\partial u_j \partial u_k} (\mathbf{u}) \Big{|} < C'.
$$
Therefore, by choosing $\delta' > 0$ sufficiently small with respect to $F$, $\mathcal{B}_0$ and $\mathcal{B}'_0$
we have $C' \delta' < M$; (\ref{ineq ofiii1}) is satisfied.

Let $\| \cdot \|$ be the $L^2$-norm on $\mathbb{R}^2$.  Let $\Phi = (\Phi_1, \ldots, \Phi_n):\mathbb{R}^n \rightarrow \mathbb{R}^n$, where
$\Phi_1(\mathbf{u}, \mathbf{v}) = \Psi_{\mathbf{v},1}(\mathbf{u})$, $\Phi_2(\mathbf{u}, \mathbf{v}) = \Psi_{\mathbf{v},2}(\mathbf{u})$,
and $\Phi_j(\mathbf{u}, \mathbf{v}) = v_j$ $(3 \leq j \leq n)$.
Then we have
$$
\det (\textnormal{Jac} \Phi(\mathbf{u}_0, \mathbf{v}_0)   ) = \mathfrak{G}_{1,2}(\mathbf{u}_0, \mathbf{v}_0) \not = 0.
$$
Therefore, as a consequence of the inverse function theorem we obtain
$$
0 < m = \min_{\mathbf{v} \in ( \mathbf{v}_0 + [- \delta', \delta']^{n-2} ) } \  \min_{\mathbf{u} \in \partial W} \| \Psi_{\mathbf{v}}(\mathbf{u}) - \Psi_{\mathbf{v}}(\mathbf{u}_0)  \|,
$$
provided $\delta' > 0$ is sufficiently small (with respect to $F$ and $\mathbf{x}_0 = (\mathbf{u}_0, \mathbf{v}_0)$);
$m > 0$ because $\| \Psi_{\mathbf{v}}(\mathbf{u}) - \Psi_{\mathbf{v}}(\mathbf{u}_0)  \|$ is continuous and it is strictly greater than $0$  on the compact
set $\partial W \times ( \mathbf{v}_0 + [- \delta', \delta']^{n-2} )$.

Given any $\mathbf{v} \in \mathcal{B}'_0$ we define
$$
V_{\mathbf{v}} = \{ \mathbf{y} \in \mathbb{R}^2 : \| \mathbf{y} - \Psi_{\mathbf{v}}(\mathbf{u}_0)  \| < m/4 \}.
$$
Let us take $\delta_1 > 0$ sufficiently small (in particular $\delta_1 < \delta'/2$) such that
$$
\| \Psi_{\mathbf{v}}(\mathbf{u})  -  \Psi_{\mathbf{v}}(\mathbf{u}_0)  \| < m/5
$$
for any $\mathbf{u}$ and $\mathbf{v}$ satisfying $\| \mathbf{u} - \mathbf{u}_0 \|_{\infty} < 2\delta_1$
and  $\| \mathbf{v} - \mathbf{v}_0 \|_{\infty} < 2\delta_1$ respectively.
Then we see that
$$
\mathcal{B}_1 \subseteq \Psi_{\mathbf{v}}^{-1} (V_{\mathbf{v}}) \subseteq W
$$
holds for any $\mathbf{v} \in \mathcal{B}'_1$.
Therefore, it follows from Theorem \ref{thm exp inv} that $\Psi |_{\mathcal{B}_1}$ is a diffeomorphism for any $\mathbf{v} \in \mathcal{B}'_1$.
\end{proof}
Let us choose  $\delta_1$ to satisfy 
\begin{eqnarray}
\label{cond1'onB1}
\delta_1 < \frac{\rho_{\min}}{2}
\end{eqnarray}
as well.
\begin{rem}
The choice of $\delta_1$ depended only on $F$, $\mathbf{x}_0$ and $\delta$.
\end{rem}
For simplicity we let
$$
G(\mathbf{u}) = F(u_1, u_2, \mathbf{v})
$$
with the understanding that the polynomial $G(\mathbf{u})$ depends on $\mathbf{v}$.
Let $\lambda_1$ and $\lambda_2$ be the smallest positive numbers satisfying
\begin{eqnarray}
\label{lambdacond2}
\Big{|} u^2_j \frac{\partial^2 G}{\partial u_j^2} (u_1, u_2)   \Big{|} + \Big{|} u_j \frac{\partial G}{\partial u_j} (u_1, u_2)   \Big{|} \leq \frac{\lambda_j}{2} \ \ (1 \leq j \leq 2)
\end{eqnarray}
for all $\mathbf{u} \in \overline{\mathcal{B}_1}$ and
$\mathbf{v} \in {\mathcal{B}'_1}$.
The reason for these choices of $\lambda_1$ and $\lambda_2$ will be clear later.
\begin{rem}
\label{rem1}
Let $A_1, A_2 \in \mathbb{R}$ and $\mathbf{v} \in \mathcal{B}'_1$. It can be verified easily that $\mathbf{z}_0 = (z_{0,1}, z_{0,2}) \in \mathcal{B}_1$ is a critical point of the function $G(\mathbf{u}) + A_1 \log u_1 + A_2 \log u_2$ if and only if
$$
\Psi_{\mathbf{v}, j}(\mathbf{z}_0) = z_{0,j} \frac{\partial G}{\partial u_j}(\mathbf{z}_0) = - A_j \ \ (1 \leq j \leq 2).
$$
Since  $\mathcal{B}_1$ is diffeomorphic to $\Psi_{\mathbf{v}}(\mathcal{B}_1)$, each pair of values
$(\Psi_{\mathbf{v},1}(\mathbf{u}), \Psi_{\mathbf{v}, 2}(\mathbf{u}))$ gets represented only once over $\mathbf{u} \in \mathcal{B}_1$.
Thus for a fixed choice of $(A_1, A_2)$ there is at most one $\mathbf{z}_0 \in \mathcal{B}_1$ for which it is a critical point
of $G(\mathbf{u}) + A_1 \log u_1 + A_2 \log u_2$ over $\mathbf{u} \in \mathcal{B}_1$.
\end{rem}

\section{Second Box - Case (I)}
\label{secsecondbox}
Fix $\mathbf{v} \in \mathcal{B}'_1$.
For $\tau \not = 0$, we have
$$
u_1^{i t_1} u_2^{i t_2} e \left( \tau F(u_1, u_2, \mathbf{v}) \right)
= e^{ i 2 \pi \tau \left( G(u_1, u_2) + \frac{t_1}{2 \pi \tau} \log u_1 + \frac{t_2}{2 \pi \tau} \log u_2  \right)}.
$$
Let
$$
A_j = \frac{t_j}{2 \pi \tau} \ \ (1 \leq j \leq 2).
$$
We now deal with the case $A_j \in [- \lambda_j, \lambda_j]$ $(1 \leq j \leq 2)$.
The critical points of the function $G(\mathbf{u}) + A_1 \log u_1 + A_2 \log u_2$ satisfy
\begin{eqnarray}
\label{crit pt condn}
\frac{\partial G}{\partial u_j} (u_1, u_2) + \frac{A_j}{u_j} = 0  \ \ \ (1 \leq j \leq 2).
\end{eqnarray}
Suppose $\mathbf{z}_0 = (z_{0,1}, z_{0,2}) \in \mathcal{B}_1$ is a critical point.
Then it follows from our choice of $\mathcal{B}_1$ and $m_0$ (defined in (\ref{defnm0m1})) that
\begin{eqnarray}
\label{m0ineq}
\Big{|} \left( \frac{\partial^{2} G}{\partial u_1^2 } (\mathbf{z}_0)  - \frac{A_1}{z_{0,1}^2}  \right) \cdot \left( \frac{\partial^{2} G}{\partial u_2^2} (\mathbf{z}_0)
- \frac{A_2}{z_{0,2}^2}
\right)
-
\left( \frac{\partial^{2} G}{\partial{u_1} \partial{u_2}} (\mathbf{z}_0) \right)^2 \Big{|}
&=&
\frac{| \mathfrak{G}_{1,2}(\mathbf{z}_0, \mathbf{v} ) | }{z_{0,1} z_{0,2} }
\\
\notag
&\geq& \frac{m_0}{\rho_{\max}^2 }.
\end{eqnarray}

Let $\widetilde{\mathcal{B}}_1 = (\mathbf{u}_0 + [-\delta_1/2, \delta_1/2]^2)$.
Suppose there exists a critical point $\mathbf{z}_0 = (z_{0,1}, z_{0,2}) \in \widetilde{\mathcal{B}}_1$, in which case we know from Remark \ref{rem1} that this is the only critical point in $\widetilde{\mathcal{B}}_1$.
Let us define
\begin{eqnarray}
\label{defn phi}
\phi(\mathbf{u}) &=& G(u_1 + z_{0,1}, u_2 + z_{0,2}) +  A_1 \log (u_1 + z_{0,1}) + A_2 \log (u_2 + z_{0,2})
\\
\notag
&-& G(z_{0,1}, z_{0,2}) - A_1 \log z_{0,1} - A_2 \log z_{0,2}.
\end{eqnarray}
We consider this function over $\mathbf{u} \in (- \mathbf{z}_0 + \widetilde{\mathcal{B}}_1) \subseteq [- \delta_1, \delta_1]^2$.
We have $\phi(\mathbf{0}) = \nabla \phi(\mathbf{0}) = \mathbf{0}$, and $\mathbf{0} \in (- \mathbf{z}_0 + \widetilde{\mathcal{B}}_1)$ is the
only critical point of $\phi(\mathbf{u})$ in $(- \mathbf{z}_0 + \widetilde{\mathcal{B}}_1)$. It follows from integration by parts that
\begin{eqnarray}
\label{eqnphi1}
\phi(\mathbf{u}) = \int_0^1 \frac{d}{dt} \phi(t \mathbf{u}) dt = \int_{0}^1 (1 - t) \frac{d^2}{dt^2} \phi(t \mathbf{u}) dt,
\end{eqnarray}
and the final expression becomes
\begin{eqnarray}
\label{eqnphi2}
\\
\notag
\int_{0}^1 (1-t)\sum_{1 \leq i,j \leq 2} u_i u_j \frac{\partial^2 G}{ \partial u_i \partial u_j }(tu_1 + z_{0,1}, tu_2 + z_{0,2}) - (1-t)\sum_{j=1}^2 A_j \frac{u_j^2}{(t u_j + z_{0,j})^2} \ dt.
\end{eqnarray}
Let $k_1, k_2 \in \mathbb{Z}_{\geq 0}$, $\mathbf{k} = (k_1, k_2)$ and $\mathbf{u}^{\mathbf{k}} = u_1^{k_1}u_2^{k_2}$.
Let us denote
$$
\frac{\partial^2 G}{ \partial u_i \partial u_j }(\mathbf{u}) = \sum_{\mathbf{k} \in [0, d]^2 } c_{i,j; \mathbf{k}}(\mathbf{v}) \  \mathbf{u}^{\mathbf{k}}.
$$
Clearly $| c_{i,j; \mathbf{k}}(\mathbf{v})| \ll 1$ for all $1 \leq i, j \leq 2$ and $\mathbf{k} \in [0, d]^2$, and the implicit constant in the bound is independent of the specific choices of $\mathbf{v} \in \mathcal{B}'_1$ and the critical point $\mathbf{z}_0 \in \widetilde{\mathcal{B}}_1$.
Then we have
\begin{eqnarray}
\frac{\partial^2 G}{ \partial u_i \partial u_j }(t u_1 + z_{0,1}, t u_2 + z_{0,2})
&=&  \sum_{\mathbf{k} \in [0, d]^2 } c_{i,j; \mathbf{k}}(\mathbf{v})  (t u_1 + z_{0,1})^{k_1} (t u_2 + z_{0,2})^{k_2}
\\
\notag
&=&
\frac{\partial^2 G}{ \partial u_i \partial u_j }(\mathbf{z}_0) + \sum_{\ell=1}^{d - 2} t^{\ell} H_{i,j}^{(\ell)} (\mathbf{u} ; \mathbf{z}_0),
\end{eqnarray}
where $H_{i,j}^{(\ell)}(\mathbf{u}; \mathbf{z}_0)$ is the homogeneous degree $\ell$ (in $\mathbf{u}$) portion
of $\frac{\partial^2 G}{ \partial u_i \partial u_j }(u_1 + z_{0,1}, u_2 + z_{0,2})$.
Note $H_{i,j}^{(\ell)}(\mathbf{u}; \mathbf{z}_0)$ may be identically $0$. It is also clear
from the definition of $H_{i,j}^{(\ell)}(\mathbf{u}; \mathbf{z}_0)$ that $H_{i,j}^{(\ell)}(\mathbf{0}; \mathbf{z}_0)$ is identically $0$.
Since $\int_0^1 (1 - t) t^{\ell} dt = \frac{1}{(\ell+1)(\ell+2)}$ $(\ell \in \mathbb{N} \cup \{0\})$, we obtain
\begin{eqnarray}
\label{eqnphi3}
&&\int_{0}^1 (1-t)  u_i u_j \frac{\partial^2 G}{ \partial u_i \partial u_j }(t u_1 + z_{0,1}, t u_2 + z_{0,2}) \ dt
\\
\notag
&=&
\frac12 u_i u_j \frac{\partial^2 G}{ \partial u_i \partial u_j }(\mathbf{z}_0) + \sum_{\ell=1}^{d-2} \frac{u_i u_j}{(\ell+1)(\ell+2)}H^{(\ell)}_{i,j} (\mathbf{u} ; \mathbf{z}_0).
\end{eqnarray}
We also have
\begin{eqnarray}
\label{eqnphi4}
&&\sum_{j=1}^2 A_j u_j^2 \int_{0}^1  \frac{1-t}{(t u_j + z_{0,j})^2}  dt
\\
&=& \sum_{j=1}^2 A_j u_j \int_{z_{0,j}}^{u_j + z_{0,j}} \frac{1}{y^2}  dy
- \sum_{j=1}^2 A_j \int_{z_{0,j} }^{u_j + z_{0,j} } \frac{y - z_{0,j} }{y^2}  dy
\notag
\\
&=& \sum_{j=1}^2 A_j \frac{u_j^2}{(u_j + z_{0,j}) z_{0,j}}
- \sum_{j=1}^2 A_j (\log (u_j + z_{0,j}) - \log z_{0,j} )
- A_j z_{0,j}  \left( \frac{-1}{u_j + z_{0,j}} - \frac{-1}{z_{0,j}}  \right)
\notag
\\
&=& \sum_{j=1}^2 A_j \frac{u_j^2}{(u_j + z_{0,j})z_{0,j}}
-  A_j \log \left( 1 + \frac{u_j}{z_{0,j}}\right)
+  A_j \frac{u_j z_{0,j}}{(u_j + z_{0,j})z_{0,j}}
\notag
\\
&=& \sum_{j=1}^2  A_j u_j^2 \sum_{\ell = 2}^{\infty}  \frac{(-1)^{\ell} u_j^{\ell - 2} }{\ell z_{0,j}^{\ell}}.
\notag
\end{eqnarray}
Note the power series in the above expression is well-defined because of (\ref{cond1'onB1}); (\ref{cond1'onB1}) implies
$$
\frac{|u_j|}{|z_{0,j}|} \leq \frac{\delta_1}{\rho_{\min}} < \frac12.
$$
Therefore, by combining (\ref{eqnphi1}), (\ref{eqnphi2}), (\ref{eqnphi3}) and (\ref{eqnphi4}) we can write
$$
\phi(\mathbf{u}) = \sum_{1 \leq i,j \leq 2} u_i u_j \phi_{i,j}(\mathbf{u}),
$$
where
$$
\phi_{i,j} (\mathbf{u})
=
\left\{
    \begin{array}{ll}
         \sum_{\ell=1}^{d - 2} \frac{H^{(\ell)}_{j,j} (\mathbf{u} ; \mathbf{z}_0)}{(\ell+1)(\ell+2)} + \frac12 \frac{\partial^2 G}{\partial u_j^2 }(\mathbf{z}_0) -
         A_j \sum_{\ell = 2}^{\infty}  \frac{(-1)^{\ell} u_j^{\ell - 2} }{\ell z_{0,j}^{\ell}}
         &\mbox{  if  } i = j ,\\
         \sum_{\ell=1}^{d - 2} \frac{H^{(\ell)}_{i,j} (\mathbf{u} ; \mathbf{z}_0)}{(\ell+1)(\ell+2)} + \frac12 \frac{\partial^2 G}{\partial u_1 \partial u_2}(\mathbf{z}_0)
         &\mbox{  if  } i \not = j.
    \end{array}
\right.
$$
\begin{rem}
\label{remark partial deriv}
For $1 \leq j \leq 2$, we have
\begin{eqnarray}
\frac{ \partial \phi_{j,j} }{ \partial u_j} (\mathbf{u})
=
\sum_{\ell=1}^{d - 2 } \frac{1}{(\ell+1)(\ell+2)} \cdot  \frac{\partial}{\partial u_j } H_{j,j}^{(\ell)} (\mathbf{u} ; \mathbf{z}_0)
-
A_j \sum_{\ell = 3}^{\infty}  \frac{(-1)^{\ell} (\ell - 2) u_j^{\ell - 3} }{\ell z_{0,j}^{\ell}}.
\end{eqnarray}
It is clear from the definition of $\phi_{i,j}$ that other first order partial derivatives of
$\phi_{i,j}$ are independent of $A_1$ and $A_2$.
\end{rem}
It can be verified that $\phi_{i,j}$ is smooth on $(- \delta_1, \delta_1)^2$, $\phi_{i,j}(\mathbf{u}) = \phi_{j,i}(\mathbf{u})$
$(\mathbf{u} \in [- \delta_1, \delta_1]^2)$ and
$$
\phi_{i,j}(0, 0) = \frac12 \cdot \frac{\partial^{2}}{\partial u_i \partial u_j}
\left( G(u_1, u_2) + A_1 \log u_1 + A_2 \log u_2 \right)  \Big{|}_{\mathbf{u} = \mathbf{z}_0}.
$$
In particular, it follows from (\ref{defm2}) and (\ref{crit pt condn}) that
\begin{eqnarray}
\label{m1ineq}
| \phi_{1,1}(0, 0) | &=&  \Big{|}  \frac12 \cdot \frac{\partial^2 G}{\partial u_1^2}(\mathbf{z}_0) -A_1 \frac{1}{2 z_{0,1}^{2}} \Big{|}
\\
\notag
&=& \frac{1}{2 z_{0,1}} \Big{|} z_{0,1} \frac{\partial^2 G}{\partial u_1^2}(\mathbf{z}_0) + \frac{\partial G}{\partial u_1}(\mathbf{z}_0) \Big{|}
\\
\notag
&\geq& \frac{m_1}{2 \rho_{\max}}
\end{eqnarray}
and
\begin{eqnarray}
\label{m1ineq'}
| \phi_{1,2}(0, 0) | =  \Big{|}  \frac12 \cdot \frac{\partial^2 G}{\partial u_1 \partial u_2}(\mathbf{z}_0) \Big{|}
\geq \frac{m_2}{2}.
\end{eqnarray}

Also there exists $m_3 > 0$ (depending only on $F$, $\widetilde{\mathcal{B}}_1$ and $\mathcal{B}'_1$) such that
\begin{eqnarray}
\label{m'1ineq}
| \phi_{1,1}(0, 0) | =  \frac{1}{2 z_{0,1} } \Big{|} z_{0,1} \frac{\partial^2 G}{\partial u_1^2}(\mathbf{z}_0) + \frac{\partial G}{\partial u_1}(\mathbf{z}_0) \Big{|} \leq \frac{  m_3  }{2 \rho_{\min}}
\end{eqnarray}
and
\begin{eqnarray}
\label{m'1ineq''}
| \phi_{1,2}(0, 0) | =  \Big{|}  \frac12 \cdot \frac{\partial^2 G}{\partial u_1 \partial u_2}(\mathbf{z}_0) \Big{|}
< \frac{m_3}{2}.
\end{eqnarray}
Since
\begin{eqnarray}
\notag
\frac{1}{4 z_{0,1} z_{0,2}} \mathfrak{G}_{1,2} (\mathbf{z}_0, \mathbf{v})
&=&
\frac{1}{4} \left( \frac{\partial^{2} G}{\partial u_1^2 } (\mathbf{z}_0)  - \frac{A_1}{z_{0,1}^2}  \right) \cdot \left( \frac{\partial^{2} G}{\partial u_2^2} (\mathbf{z}_0)
- \frac{A_2}{z_{0,2}^2}
\right)
-
\frac{1}{4} \left( \frac{\partial^{2} G}{\partial{u_1} \partial{u_2}} (\mathbf{z}_0) \right)^2
\\
\notag
&=& \phi_{1,1}(\mathbf{0}) \phi_{2,2}(\mathbf{0}) - \phi_{1,2}^2(\mathbf{0}),
\end{eqnarray}
it is easy to see that there exists $m_3'>0$ (depending only on $F$, $\widetilde{\mathcal{B}}_1$ and $\mathcal{B}'_1$) such that
\begin{eqnarray}
\label{m'1ineq'}
|\phi_{1,1}(\mathbf{0}) \phi_{2,2}(\mathbf{0}) - \phi_{1,2}^2(\mathbf{0})| < \frac{m_3'}{4 \rho_{min}^2}.
\end{eqnarray}

\begin{rem}
\label{rem2}
We have
\begin{eqnarray}
\notag
&& \Big{|} \phi_{1,1}(\mathbf{u}) \phi_{2,2}(\mathbf{u}) - \phi_{1,2}^2(\mathbf{u}) -
\frac{1}{4} \left( \frac{\partial^{2} G}{\partial u_1^2 } (\mathbf{z}_0)  - \frac{A_1}{z_{0,1}^2}  \right) \cdot \left( \frac{\partial^{2} G}{\partial u_2^2} (\mathbf{z}_0)
- \frac{A_2}{z_{0,2}^2}
\right)
+
\frac{1}{4} \left( \frac{\partial^{2} G}{\partial{u_1} \partial{u_2}} (\mathbf{z}_0) \right)^2 \Big{|}
\label{rem1ineq1}
\\
&\leq&
\notag
C_1  \  \max_{1 \leq j \leq 2} |u_j|  \  \  \  ( - \delta_1/2 \leq u_1, u_2 \leq \delta_1/2),
\end{eqnarray}
where $C_1 > 0$ depends on $F$, $\delta_1$, $\lambda_1, \lambda_2$, $\widetilde{\mathcal{B}}_1$, $\mathcal{B}'_1$,
$$
\max_{\substack{\mathbf{z} \in  \widetilde{\mathcal{B}}_1  \\ \mathbf{v} \in \mathcal{B}_1' \\  A_1 \in [- \lambda_1,  \lambda_1] }}  \Big{|} \frac{\partial^{2} G}{\partial u_1^2 } (\mathbf{z})  - \frac{A_1}{z_1^2} \Big{|},
\max_{\substack{\mathbf{z} \in \widetilde{\mathcal{B}}_1 \\ \mathbf{v} \in \mathcal{B}_1' \\  A_2 \in [- \lambda_2, \lambda_2] }}
\Big{|}\frac{\partial^{2} G}{\partial u_2^2} (\mathbf{z})
- \frac{A_2}{z_2^2} \Big{|}  \ \
\text{  and  }  \   \
\max_{\substack{\mathbf{z} \in \widetilde{\mathcal{B}}_1 \\ \mathbf{v} \in \mathcal{B}_1' } }\Big{|} \frac{\partial^{2} G}{\partial{u_1} \partial{u_2}} (\mathbf{z}) \Big{|}.
$$
Similarly, we have
\begin{eqnarray}
\notag
\Big{|} \phi_{1,1}(\mathbf{u}) -
\frac{1}{2} \left( \frac{\partial^{2} G}{\partial u_1^2 } (\mathbf{z}_0)  - \frac{A_1}{z_{0,1}^2}  \right) \Big{|}
\leq
C_2  \  \max_{1 \leq j \leq 2} |u_j|   \  \  \   (- \delta_1/2 \leq u_1, u_2 \leq \delta_1/2)
\end{eqnarray}
and
\begin{eqnarray}
\notag
\Big{|} \phi_{1,2}(\mathbf{u}) -
\frac{1}{2} \left( \frac{\partial^{2} G}{\partial u_1 \partial u_2 } (\mathbf{z}_0) \right) \Big{|}
\leq
C_2  \  \max_{1 \leq j \leq 2} |u_j|   \  \  \   (- \delta_1/2 \leq u_1, u_2 \leq \delta_1/2),
\end{eqnarray}
where $C_2 > 0$ depends on $F$, $\delta_1$, $\lambda_1, \lambda_2$, $\widetilde{\mathcal{B}}_1$ and $\mathcal{B}'_1$.
In particular, both $C_1$ and $C_2$ are independent of the specific choices of $\mathbf{v} \in \mathcal{B}'_1$, $A_1 \in [- \lambda_1, \lambda_1]$, $A_2 \in [- \lambda_2, \lambda_2]$, and the critical point $\mathbf{z}_0 \in \widetilde{\mathcal{B}}_1$.
\end{rem}
It follows from Remark \ref{rem2}, (\ref{m0ineq}), (\ref{m1ineq}) and (\ref{m1ineq'}) that
we can find $\delta_4 > 0$ (in particular satisfying $\delta_4 < \delta_1/2$) and $m_4 > 0$  (both values are independent of the specific choices of $\mathbf{v} \in \mathcal{B}'_1$,  $A_1 \in [- \lambda_1, \lambda_1]$, $A_2 \in [- \lambda_2, \lambda_2]$, and the critical point $\mathbf{z}_0 \in \widetilde{\mathcal{B}}_1$) such that
\begin{eqnarray}
\label{rem1ineq3}
| \phi_{1,1}(\mathbf{u}) \phi_{2,2}(\mathbf{u}) - \phi_{1,2}^2(\mathbf{u}) | >
m_4 \ \ (\mathbf{u} \in [-\delta_4, \delta_4]^2),
\end{eqnarray}
\begin{eqnarray}
\label{rem1ineq4}
| \phi_{1,1}(\mathbf{u}) | > m_4 \ \ (\mathbf{u} \in [-\delta_4, \delta_4]^2) \ \ \text{ and } \ \ | \phi_{1,2}(\mathbf{u}) | > m_4 \ \ (\mathbf{u} \in [-\delta_4, \delta_4]^2).
\end{eqnarray}

Let $\varepsilon_1$ be the sign of $\phi_{1,1}(\mathbf{u}) $ over $\mathbf{u} \in [-\delta_4, \delta_4]^2$
and let $\varepsilon_2$ be the sign of $\phi_{1,1}(\mathbf{u}) \phi_{2,2}(\mathbf{u}) - \phi_{1,2}^2(\mathbf{u})$ over $\mathbf{u} \in [-\delta_4, \delta_4]^2$.
We define a new set of variables $\mathbf{y} = (y_1, y_2)$ by
$$
y_1 = \mathcal{F}_1(u_1, u_2)   \ \ \text{  and  }  \  \  y_2 = \mathcal{F}_2(u_1, u_2),
$$
where
$$
\mathcal{F}_1(u_1, u_2) =  \sqrt{ \varepsilon_1 \phi_{1,1}(\mathbf{u})}  \left( u_1 + \frac{\phi_{1,2}(\mathbf{u})}{\phi_{1,1}(\mathbf{u})} u_2 \right)
$$
and
$$
\mathcal{F}_2(u_1, u_2) = u_2   \sqrt{ \varepsilon_2  \frac{  \phi_{1,1}(\mathbf{u}) \phi_{2,2}(\mathbf{u}) - \phi_{1,2}^2(\mathbf{u})  }{
 \varepsilon_1 \phi_{1,1}(\mathbf{u}) } }.
$$
Let $\mathcal{F} = (\mathcal{F}_1, \mathcal{F}_2)$. It can be verified easily that
$$
\textnormal{Jac}\mathcal{F}(\mathbf{0}) =
\begin{pmatrix}
\sqrt{\varepsilon_1 \phi_{1,1}(\mathbf{0})}  &  \frac{\phi_{1,2}(\mathbf{0} ) \sqrt{\varepsilon_1 \phi_{1,1}(\mathbf{0})}  }{ \phi_{1,1}(\mathbf{0}) }
\\ 0 &
\sqrt{ \varepsilon_2  \frac{  \phi_{1,1}(\mathbf{0}) \phi_{2,2}(\mathbf{0}) - \phi_{1,2}^2(\mathbf{0})   }{  \varepsilon_1  \phi_{1,1}(\mathbf{0}) } }
\end{pmatrix}
$$
and $\det (\textnormal{Jac}\mathcal{F}(\mathbf{0})) \not = 0$.

Let
$$
M =
\frac{1}{8} \sqrt{m_4}
\left( \sqrt{  \frac{m_3}{ 2 \rho_{\min} }} + \frac{m_3  }{ 2 \sqrt{m_4}}  + \frac{\sqrt{m'_3} }{2 \rho_{\min} \sqrt{m_4} } \right)^{-1}.
$$
It then follows from  (\ref{m'1ineq}), (\ref{m'1ineq''}), (\ref{m'1ineq'}) and (\ref{rem1ineq4}) that
$$
0 < M <  \frac{|\det (\textnormal{Jac}\mathcal{F}(\mathbf{0}) )|}{ 2 \cdot 2! \cdot   \max_{1 \leq i, j \leq 2}  | [\textnormal{Jac}\mathcal{F}(\mathbf{0})]_{i,j}  |}.
$$

Let $W = (- \delta_5, \delta_5)^2$ with $\delta_5 > 0$ sufficiently small.
First we make sure $W$ satisfies the three properties of Theorem \ref{thm exp inv} with respect to $\mathcal{F}$.
In the process we also show that $\delta_5$ can be chosen independently of the specific choices of our parameters.

i) It is clear from the definition that $\mathbf{0} \in W$.

ii) Let us prove $\textnormal{det}\left( \textnormal{Jac}\mathcal{F}(\mathbf{u}) \right) \not = 0$ $(\mathbf{u} \in W)$.
We have
\begin{eqnarray}
\notag
[\textnormal{Jac}\mathcal{F}(\mathbf{u})]_{1,1} =
\left( \frac{\partial}{\partial u_1} \sqrt{ \varepsilon_1 \phi_{1,1}} \right)
\left( u_1 + u_2 \frac{\phi_{1,2}}{\phi_{1,1}} \right)
+
\sqrt{\varepsilon_1 \phi_{1,1}} \left( 1 + u_2  \frac{\partial}{\partial u_1} \  \frac{\phi_{1,2}}{\phi_{1,1}}  \right),
\end{eqnarray}
\begin{eqnarray}
\notag
[\textnormal{Jac}\mathcal{F}(\mathbf{u})]_{1,2} =
\left( \frac{\partial}{\partial u_2} \sqrt{ \varepsilon_1 \phi_{1,1}} \right)
\left( u_1 + u_2 \frac{\phi_{1,2}}{\phi_{1,1}} \right)
+
\sqrt{\varepsilon_1 \phi_{1,1}} \left( \frac{\phi_{1,2}}{\phi_{1,1}} + u_2  \frac{\partial}{\partial u_2} \  \frac{\phi_{1,2}}{\phi_{1,1}}  \right),
\end{eqnarray}
\begin{eqnarray}
\notag
[\textnormal{Jac}\mathcal{F}(\mathbf{u})]_{2,1} =
u_2 \frac{\partial}{\partial u_1} \sqrt{   \frac{  \varepsilon_2( \phi_{1,1} \phi_{2,2} - \phi_{1,2}^2 ) }{
\varepsilon_1 \phi_{1,1} } }
\end{eqnarray}
and
\begin{eqnarray}
\notag
[\textnormal{Jac}\mathcal{F}(\mathbf{u})]_{2,2} =
\sqrt{   \frac{  \varepsilon_2( \phi_{1,1} \phi_{2,2} - \phi_{1,2}^2 ) }{
\varepsilon_1 \phi_{1,1} } }
+
u_2 \frac{\partial}{\partial u_2} \sqrt{   \frac{  \varepsilon_2( \phi_{1,1} \phi_{2,2} - \phi_{1,2}^2 ) }{
\varepsilon_1 \phi_{1,1} } }.
\end{eqnarray}
Thus it can be verified that there exists $C_3 > 0$ such that
\begin{eqnarray}
\label{eqn4.17}
\det(\textnormal{Jac}\mathcal{F}(\mathbf{u})) - \det(\textnormal{Jac}\mathcal{F}(\mathbf{0})) \leq  C_3 \max_{1 \leq j \leq 2} |u_j| \ \ \ (\mathbf{u} \in
[- \delta_4, \delta_4]^2  ),
\end{eqnarray}
where $C_3$ is independent of the specific choices of $\mathbf{v} \in \mathcal{B}'_1$,
$A_1 \in [- \lambda_1, \lambda_1]$, $A_2 \in [- \lambda_2, \lambda_2]$, and the critical point $\mathbf{z}_0 \in \widetilde{\mathcal{B}}_1$.
Therefore, it follows from  (\ref{rem1ineq3}) and (\ref{eqn4.17}) that
\begin{eqnarray}
\label{lowerbounddet}
| \det(\textnormal{Jac}\mathcal{F}(\mathbf{u})) | >  \frac{ \sqrt{m_4} }{2} \ \ \ (\mathbf{u} \in
[- \delta_5, \delta_5]^2  )
\end{eqnarray}
for $\delta_5$ sufficiently small.
\begin{rem} \label{rem 4.3}
It can be verified that $C_3 > 0$ depends only on $\delta_4$, $m_4$ (from (\ref{rem1ineq3}) and (\ref{rem1ineq4})),
\begin{eqnarray}
\label{twomax2}
\max_{1 \leq i,j \leq 2}
\max_{\substack{ A_1 \in [- \lambda_1, \lambda_1] \\ A_2 \in [- \lambda_2, \lambda_2] }}
\max_{\substack{\mathbf{u} \in [- \delta_4, \delta_4]^2 \\  \mathbf{v} \in \mathcal{B}'_1  \\ \mathbf{z}_0 \in \widetilde{\mathcal{B}}_1 }}  |\phi_{i,j} (\mathbf{u})|,
\ \
\max_{1 \leq i,j, k \leq 2}
\max_{\substack{A_1 \in [- \lambda_1, \lambda_1] \\ A_2 \in [- \lambda_2, \lambda_2] }}
\max_{\substack{\mathbf{u} \in [- \delta_4, \delta_4]^2 \\  \mathbf{v} \in \mathcal{B}'_1 \\ \mathbf{z}_0 \in \widetilde{\mathcal{B}}_1 }} \Big{|} \frac{\partial \phi_{i,j} }{\partial u_k}  (\mathbf{u}) \Big{|},
\end{eqnarray}
and
\begin{eqnarray}
\label{twomax2'}
\max_{1 \leq i,j,k, \ell \leq 2}
\max_{\substack{ A_1 \in [- \lambda_1, \lambda_1] \\ A_2 \in [- \lambda_2, \lambda_2] }}
\max_{\substack{\mathbf{u} \in [- \delta_4, \delta_4]^2 \\  \mathbf{v} \in \mathcal{B}'_1 \\ \mathbf{z}_0 \in \widetilde{\mathcal{B}}_1 }}  \Big{|} \frac{\partial^2 \phi_{i,j} }{ \partial u_k \partial u_{\ell} }  (\mathbf{u}) \Big{|}.
\end{eqnarray}
Since the definition of the function $\phi_{i,j} (\mathbf{u})$ doesn't require $\mathbf{z}_0$ to be a critical point of
$G(\mathbf{u}) + A_1 \log u_1 + A_2 \log u_2$, in the above expression the maximum over $\mathbf{z}_0$ is taken over all points in $\widetilde{\mathcal{B}}_1$; we used the fact that
$\mathbf{z}_0$ is a critical point to obtain lower bounds, but this is not necessary for upper bounds.
\end{rem}

iii) We now prove
\begin{eqnarray}
\label{4.19'}
\Big{|} \frac{\partial \mathcal{F}_i}{ \partial u_j} (\mathbf{u}) - \frac{\partial \mathcal{F}_i}{ \partial u_j} (\mathbf{0}) \Big{|} < M \ \ \ (\mathbf{u} \in W, \ 1 \leq i, j \leq n).
\end{eqnarray}
Given $\mathbf{u} \in [- \delta_4, \delta_4]^2$ it follows from the mean value theorem that there exists $\mathbf{c}' \in [- \delta_4, \delta_4]^2$
such that
\begin{eqnarray}
\label{meanvalthm1}
\frac{\partial \mathcal{F}_i}{ \partial u_j} (\mathbf{u}) - \frac{\partial \mathcal{F}_i}{ \partial u_j} (\mathbf{0})
=
\nabla  \frac{\partial \mathcal{F}_i}{ \partial u_j} (\mathbf{c}') \cdot \mathbf{u}.
\end{eqnarray}
It can be verified that there exists $C_4 > 0$ which depends only on $\delta_4$, $m_4$ (from (\ref{rem1ineq3}) and (\ref{rem1ineq4})), and
quantities in (\ref{twomax2}) and (\ref{twomax2'})
such that
\begin{eqnarray}
\label{bdd on sec der F}
\max_{ \substack{1 \leq i,j, k \leq 2  }  } \
\sup_{\substack{  \mathbf{v},  A_1, A_2 , \mathbf{z}_0 } } \
\max_{\mathbf{c} \in [- \delta_4, \delta_4]^2 } \Big{|}  \frac{\partial^2 \mathcal{F}_i}{ \partial u_j \partial u_k} (\mathbf{c})  \Big{|} < C_4,
\end{eqnarray}
where the supremum is taken over all $\mathbf{v} \in \mathcal{B}'_1$, $A_1 \in [- \lambda_1, \lambda_1]$,
$A_2 \in [- \lambda_2, \lambda_2]$ (for which $G(\mathbf{u}) + A_1 \log u_1 + A_2 \log u_2$ has a critical point in
$\widetilde{\mathcal{B}}_1$), and the critical point $\mathbf{z}_0 \in \widetilde{\mathcal{B}}_1$.
Thus, we obtain from (\ref{meanvalthm1}) and (\ref{bdd on sec der F}) that
$$
\Big{|} \frac{\partial \mathcal{F}_i}{ \partial u_j} (\mathbf{u}) - \frac{\partial \mathcal{F}_i}{ \partial u_j} (\mathbf{0}) \Big{|}
< 2 C_4 \max_{1 \leq j \leq 2} | u_j | \ \ \ (\mathbf{u} \in [- \delta_4, \delta_4]^2).
$$
Therefore, for $\delta_5$ sufficiently small we have (\ref{4.19'}).
Let us fix a choice of $\delta_5 < \delta_4$ that satisfies i), ii) and iii).
\begin{rem}
The choice of $\delta_5$ depended only on $C_3$, $\delta_4$ and $m_4$ in ii),
and only on $C_4$, $\delta_4$ and $M$ in iii). Therefore,
$\delta_5$ can be chosen independently of the specific choices of
$\mathbf{v} \in \mathcal{B}'_1$, $A_1 \in [- \lambda_1, \lambda_1]$, $A_2 \in [- \lambda_2, \lambda_2]$,
and the critical point $\mathbf{z}_0 \in \widetilde{\mathcal{B}}_1$.
\end{rem}

\begin{claim} \label{claim2} There exists $m_5 > 0$ such that the inequality
$$
m_5 \leq \min_{\mathbf{u} \in \partial W} \| \mathcal{F}(\mathbf{u}) - \mathcal{F}(\mathbf{0}) \|
$$
holds uniformly over all choices of $\mathbf{v} \in \mathcal{B}_1'$, $A_1 \in [- \lambda_1, \lambda_1]$,
$A_2 \in [- \lambda_2, \lambda_2]$ (for which $G(\mathbf{u}) + A_1 \log u_1 + A_2 \log u_2$ has a critical point in
$\widetilde{\mathcal{B}}_1$), and the critical point $\mathbf{z}_0 \in \widetilde{\mathcal{B}}_1$.
\end{claim}
\begin{proof}
First note $\mathcal{F}(\mathbf{0}) = ( \mathcal{F}_1(0), \mathcal{F}_2(0) ) = \mathbf{0}$.
On the top and bottom edges of the square $\partial W$ (points with $u_2 = \pm \delta_5$), we have
\begin{eqnarray}
|\mathcal{F}_2(u_1, u_2)| >  \delta_5  \frac{ \sqrt{m_4} }{ \sqrt{M'} },
\end{eqnarray}
where
$$
0 < m_4 \leq M' =
\sup_{\substack{  \mathbf{v},  A_1, A_2 , \mathbf{z}_0 } } \
\max_{\mathbf{u} \in \partial W} \
( |\phi_{1,1}(\mathbf{u})| +  |\phi_{1,2}(\mathbf{u})| )
$$
and the supremum is taken over all $\mathbf{v} \in \mathcal{B}'_1$, $A_1 \in [- \lambda_1, \lambda_1]$,
$A_2 \in [- \lambda_2, \lambda_2]$ (for which $G(\mathbf{u}) + A_1 \log u_1 + A_2 \log u_2$ has a critical point in
$\widetilde{\mathcal{B}}_1$), and the critical point $\mathbf{z}_0 \in \widetilde{\mathcal{B}}_1$.
Note it can be easily shown that $M'$ is a bounded quantity by the definitions of $\phi_{1,1}$ and $\phi_{1,2}$.

We now treat the left and right edges of the square $\partial W$ (points with $u_1 = \pm \delta_5$).
Let $\mathbf{u}$ be on the left or the right edge of $\partial W$.
If $|u_2| > \frac{\delta_5  m_4}{2 M'}$
then we have
$$
|\mathcal{F}_2(u_1, u_2)| >  \frac{\delta_5  m_4}{2 M'}  \cdot  \frac{ \sqrt{m_4} }{ \sqrt{M'} } > 0.
$$
On the other hand,  if $|u_2| \leq \frac{\delta_5  m_4}{2 M'}$
then we have
$$
\Big{|} u_1 + \frac{\phi_{1,2}(\mathbf{u})}{\phi_{1,1}(\mathbf{u})} u_2 \Big{|} \geq \frac{ \delta_5 }{2},
$$
and consequently
$$
|\mathcal{F}_1(u_1, u_2)| > \sqrt{m_4}   \frac{ \delta_5 }{2} > 0.
$$
Since $\| \mathcal{F}(\mathbf{u}) \|^2 = |\mathcal{F}_1(\mathbf{u})|^2 + |\mathcal{F}_2(\mathbf{u})|^2$
we see that our claim holds.
\end{proof}

\begin{claim} \label{claim3} There exists $\delta_6 > 0$ such that the inequality
$$
\| \mathcal{F}(\mathbf{u})   \|  < \frac{m_5}{4} \ \ \ (\mathbf{u} \in [- 4 \delta_6, 4 \delta_6]^2)
$$
holds uniformly over all choices of
$\mathbf{v} \in \mathcal{B}'_1$,
$A_1 \in [- \lambda_1, \lambda_1]$, $A_2 \in [- \lambda_2, \lambda_2]$
(for which $G(\mathbf{u}) + A_1 \log u_1 + A_2 \log u_2$ has a critical point in
$\widetilde{\mathcal{B}}_1$),
and the critical point $\mathbf{z}_0 \in \widetilde{\mathcal{B}}_1$.
\end{claim}
\begin{proof} We have $\mathcal{F}(\mathbf{0}) = (\mathcal{F}_1(\mathbf{0}), \mathcal{F}_2(\mathbf{0})) = \mathbf{0}$.
By a similar calculation as in ii), it can be verified that the absolute value of each entry of the matrix $\textnormal{Jac}\mathcal{F}(\mathbf{u})$
can be bounded (from above) uniformly over $\mathbf{u} \in [- \delta_4, \delta_4]^2$, $\mathbf{v} \in \mathcal{B}'_1$, $A_1 \in [- \lambda_1, \lambda_1]$,
$A_2 \in [- \lambda_2, \lambda_2]$ (for which $G(\mathbf{u}) + A_1 \log u_1 + A_2 \log u_2$ has a critical point in
$\widetilde{\mathcal{B}}_1$),  and the critical point $\mathbf{z}_0 \in \widetilde{\mathcal{B}}_1$.
Thus from the mean value theorem it follows that there exists $C_5 >0$ such that
$$
| \mathcal{F}_i(\mathbf{u})  | \leq C_5  \max_{1 \leq j \leq 2} | u_j | \ \ \  (\mathbf{u} \in [- \delta_4, \delta_4]^2).
$$
Therefore, there exists $\delta_6 > 0$, depending only on $C_5$, $\delta_4$ and $m_5$, such that
$$
\| \mathcal{F}(\mathbf{u})   \|  < \frac{m_5}{4} \ \ \ (\mathbf{u} \in  [- 4 \delta_6, 4 \delta_6]^2).
$$
\end{proof}
Let us fix $\delta_6 > 0$ sufficiently small, in particular $\delta_6 < \frac15 \min \{ \delta_1, \delta_5 \}$.
Let
$$
V = \{ \mathbf{y} \in \mathbb{R}^2 :  \| \mathbf{y} \| < m_5 /2  \}.
$$
Then we have $[- 3\delta_6, 3 \delta_6]^2 \subseteq \mathcal{F}^{-1}(V)$; in fact, by Theorem \ref{thm exp inv} $\mathcal{F}$ is a diffeomorphism from $(- 4 \delta_6, 4 \delta_6)^2$ to an open subset of $V$.
Let $\mathcal{B}_2 = (\mathbf{u}_0 + [-\delta_6, \delta_6]^2)$. We have
$\mathcal{B}_2 \subseteq \widetilde{\mathcal{B}}_1$, and
given any $\mathbf{z} \in \mathcal{B}_2$,
\begin{eqnarray}
\label{delta7}
\mathcal{B}_2 \subseteq (\mathbf{z} + [-3\delta_6, 3\delta_6]^2) \subseteq \mathcal{B}_1
\end{eqnarray}
and
\begin{eqnarray}
\label{bdd set of B2one}
(- \mathbf{z} + \mathcal{B}_2) \subseteq [-3\delta_6, 3\delta_6]^2.
\end{eqnarray}
Therefore, it follows that given any $\mathbf{v} \in \mathcal{B}_1'$, $A_1 \in [- \lambda_1, \lambda_1]$ and $A_2 \in [- \lambda_2, \lambda_2]$,
and if $G(\mathbf{u}) + A_1 \log u_1 + A_2 \log u_2$ has a critical point $\mathbf{z}_0 \in \mathcal{B}_2$, then
$\mathcal{F} |_{- \mathbf{z}_0 + \mathcal{B}_2}$ is a diffeomorphism (It is a restriction of a diffeomorphism defined on $(-4 \delta_6, 4 \delta_6)^2$
to a closed subset $(- \mathbf{z}_0 + \mathcal{B}_2) \subseteq  (-4 \delta_6, 4 \delta_6)^2$.).

\begin{rem}
The main reference for the material in this section was the proof of \cite[Chapter VIII, Proposition 6]{St},
where an estimate is obtained by the stationary phase method. 
\end{rem}

\section{Third Box - Case (I)}
\label{secthirdbox}
Let $A_j \in [- \lambda_j, \lambda_j]$ $(1 \leq j \leq 2)$.
Recall the definition of $\Psi_{\mathbf{v}}$ from Section \ref{secfirstbox}.
Let $\eta_0 > 0$ be such that
\begin{eqnarray}
\label{PhiB2 contain}
(\Psi_{\mathbf{v}_0}(\mathbf{u}_0) + [-7 \eta_0, 7 \eta_0]^2)  \subseteq \Psi_{\mathbf{v}_0}(\mathcal{B}_2).
\end{eqnarray}
Since $\Psi_{\mathbf{v}}(\mathbf{u})$ is a polynomial in $\mathbf{u}$ and $\mathbf{v}$, it is easy to deduce that there exists $\widetilde{\delta} > 0$ such that
\begin{eqnarray}
\label{etabdd1}
\|  \Psi_{\mathbf{v}}(\mathbf{u}) - \Psi_{\mathbf{v}_0}(\mathbf{u}) \|_{\infty} < \eta_0
\end{eqnarray}
for all $\mathbf{u} \in \mathcal{B}_2$ and $\mathbf{v}$ satisfying $\| \mathbf{v} - \mathbf{v}_0 \|_{\infty} < 2 \widetilde{\delta}$;
the choice of $\widetilde{\delta}$ depends only on $F, \mathcal{B}_2, \mathbf{v}_0$ and $\eta_0$.

Recall by Claim \ref{claim1} $\Psi_{\mathbf{v}}$ is a diffeomorphism on $\mathcal{B}_1$  for any $\mathbf{v} \in \mathcal{B}'_1$, and also
$\mathcal{B}_2 \subseteq \mathcal{B}_1$. Let $\mathcal{B}'_2 = (\mathbf{v}_0 + [- \min \{ \widetilde{\delta}, \delta_1\}, \min \{ \widetilde{\delta}, \delta_1\}]^2 )$ and $\mathbf{v} \in \mathcal{B}'_2$.
Let
$$
\mathcal{Y}_0 = ( \Psi_{\mathbf{v}_0}(\mathbf{u}_0) + [-6 \eta_0, 6 \eta_0]^2 )  \subseteq \Psi_{\mathbf{v}_0}(\mathcal{B}_2)
$$
and
$$
\mathcal{Y}_1 = \Psi_{\mathbf{v}}(  \Psi_{\mathbf{v}_0}^{-1}(\mathcal{Y}_0)).
$$
Let $\mathcal{C}$ be the boundary of $\Psi_{\mathbf{v}_0}^{-1}(\mathcal{Y}_0)$, i.e. $\mathcal{C} = \partial \Psi_{\mathbf{v}_0}^{-1}(\mathcal{Y}_0).$
Note since $\Psi_{\mathbf{v}_0}^{-1}(\mathcal{Y}_0) \subseteq \mathcal{B}_2$ (and $\Psi_{\mathbf{v}_0}^{-1}(\mathcal{Y}_0)$ is closed) we have $\mathcal{C} \subseteq \mathcal{B}_2$. By basic properties of diffeomorphisms it follows that $\Psi_{\mathbf{v}_0}$ maps the boundary of $\Psi_{\mathbf{v}_0}^{-1}(\mathcal{Y}_0)$ to the boundary of $\mathcal{Y}_0$, and $\Psi_{\mathbf{v}}$ maps the boundary of $\Psi_{\mathbf{v}_0}^{-1}(\mathcal{Y}_0)$ to the boundary of $\mathcal{Y}_1$; in other words,  $\Psi_{\mathbf{v}_0}(\mathcal{C})  = \partial \mathcal{Y}_0$ and $\Psi_{\mathbf{v}}(\mathcal{C}) = \partial \mathcal{Y}_1$.
Recall diffeomorphims map non-self intersecting loops to non-self intersecting loops.
Since $\Psi_{\mathbf{v}_0}^{-1}$ is a diffeomorphism and $\partial \mathcal{Y}_0$ is a non-self intersecting loop,
we have that $\mathcal{C}$ is also a non-self intersecting loop.
Consequently, $\partial \mathcal{Y}_1$ is a non-self intersecting loop as well.

Take any $\mathbf{s} \in \partial \mathcal{Y}_1$. Then there exists $\mathbf{u} \in \mathcal{C}$ such that
$\mathbf{s} = \Psi_{\mathbf{v}}(\mathbf{u})$. By (\ref{etabdd1}) we have
$$
\| \mathbf{s} - \Psi_{\mathbf{v}_0}(\mathbf{u}) \|_{\infty} = \| \Psi_{\mathbf{v}}( \mathbf{u}) - \Psi_{\mathbf{v}_0}( \mathbf{u})  \|_{\infty} < \eta_0,
$$
and we know $\Psi_{\mathbf{v}_0}( \mathbf{u}) \in \partial \mathcal{Y}_0$.
Therefore, it follows that
\begin{eqnarray}
\label{propY0}
\partial \mathcal{Y}_1 \subseteq \cup_{\mathbf{s}' \in \partial \mathcal{Y}_0 }  \   (\mathbf{s}' + (- \eta_0, \eta_0)^2 )
= (\partial \mathcal{Y}_0) + (- \eta_0, \eta_0)^2.
\end{eqnarray}
Since
$\Psi_{\mathbf{v}_0}^{-1}$ and $\Psi_{\mathbf{v}}$ are diffeomorphisms and
$\mathcal{Y}_0$ is simply connected, $\mathcal{Y}_1$ is also simply connected. Thus we obtain
from the definition of $\mathcal{Y}_0$ and (\ref{propY0}) that
\begin{eqnarray}
\label{3eta0}
(\Psi_{\mathbf{v}_0}(\mathbf{u}_0) + [- 4  \eta_0, 4 \eta_0]^2)  \subseteq \mathcal{Y}_1 \subseteq \Psi_{\mathbf{v}}(\mathcal{B}_2) \ \ \ \ (\mathbf{v} \in \mathcal{B}'_2).
\end{eqnarray}

Let $0 < \delta_7 < \min \{ \widetilde{\delta}, \delta_1\}$ be sufficiently small such that
\begin{eqnarray}
\label{choicedelta7}
\| \Psi_{\mathbf{v}}(\mathbf{u})  -  \Psi_{\mathbf{v}_0}(\mathbf{u}_0)  \|_{\infty} < \eta_0 / 2
\end{eqnarray}
for all $\mathbf{u}$ and $\mathbf{v}$ satisfying $\| \mathbf{u} - \mathbf{u}_0 \|_{\infty} < 2\delta_7$ and
$\| \mathbf{v} - \mathbf{v}_0 \|_{\infty} < 2 \delta_7$ respectively.
Let $\mathcal{B}_3 = (\mathbf{u}_0 + [- \delta_7, \delta_7]^2)$ and
$\mathcal{B}'_3 = (\mathbf{v}_0 + [- \delta_7, \delta_7]^{n-2})$.

Let us fix $\mathbf{v} \in \mathcal{B}'_3$.
Let $\mathfrak{D}_1 = (\Psi_{\mathbf{v}}(\mathbf{u}_0) + [-\eta_0, \eta_0]^2)$ and $\mathfrak{D}_2 = (\Psi_{\mathbf{v}}(\mathbf{u}_0) + [- 3\eta_0, 3\eta_0]^2)$.
Then it follows from (\ref{choicedelta7}) that $\Psi_{\mathbf{v}}(\mathcal{B}_3) \subseteq \mathfrak{D}_1$, and from (\ref{3eta0}) that
$$
\mathfrak{D}_2 \subseteq \Psi_{\mathbf{v}}(\mathcal{B}_2).
$$
Let $\mathbf{w} \in \mathbb{R}^2$. Clearly if $(\mathbf{w} + [-\eta_0, \eta_0]^2) \cap \mathfrak{D}_1 \not = \emptyset$, then
\begin{eqnarray}
\label{wshiftbox}
\left( \mathbf{w} + [-\eta_0, \eta_0]^2 \right) \cup \mathfrak{D}_1 \subseteq \mathfrak{D}_2.
\end{eqnarray}

Suppose $(- A_1, - A_2) \not \in (\mathfrak{D}_1 - [-\eta_0, \eta_0]^2)$. Then
we have
\begin{eqnarray}
\label{star'}
[-\eta_0, \eta_0]^2 \cap \left( (A_1, A_2) + \mathfrak{D}_1  \right)  = \emptyset.
\end{eqnarray}
Thus for any $(s_1,s_2) \in \mathfrak{D}_1$ at least one of  $|s_1 + A_1| > \eta_0$ or $|s_2 + A_2| > \eta_0$ holds.
We now prove either $|s_1 + A_1| > \eta_0$ holds for all $(s_1, s_2) \in \mathfrak{D}_1$
or $|s_2 + A_2| > \eta_0$ holds for all $(s_1, s_2) \in \mathfrak{D}_1$.
Suppose otherwise, in which case there exist $(s_1, s_2) \in \mathfrak{D}_1$ with
$|s_1 + A_1| \leq \eta_0$ and $(s'_1, s'_2) \in \mathfrak{D}_1$ with $|s'_2 + A_2| \leq \eta_0$.
Since $\mathfrak{D}_1$ is a square in $\mathbb{R}^2$ (with sides parallel to the axes)
it follows that $(s_1, s'_2) \in \mathfrak{D}_1$, and this contradicts (\ref{star'}).
Therefore, we have either $|s_1 + A_1| > \eta_0$ holds for all $(s_1, s_2) \in \mathfrak{D}_1$
or $|s_2 + A_2| > \eta_0$ holds for all $(s_1, s_2) \in \mathfrak{D}_1$.
It follows that there exists $1 \leq j_0 \leq 2$ such that 
$$
| \Psi_{\mathbf{v}, j_0}(\mathbf{u}) + A_{j_0}  | = \Big{|} u_{j_0} \frac{\partial G}{\partial u_{j_0} }(\mathbf{u}) + A_{j_0} \Big{|} > \eta_0 \ \ (\mathbf{u} \in \Psi_{\mathbf{v}}^{-1}(\mathfrak{D}_1) ),
$$
and consequently
\begin{eqnarray}
\label{deriv bound1}
\Big{|} \frac{\partial G}{\partial u_{j_0} }(\mathbf{u}) +  \frac{A_{j_0}}{u_{j_0}} \Big{|} \geq \frac{\eta_0}{\rho_{\max}}
\ \ (\mathbf{u} \in \mathcal{B}_3).
\end{eqnarray}
Thus we have a uniform lower bound on at least one of the first partial derivatives of $G(\mathbf{u}) + A_1 \log u_1 + A_2 \log u_2$ in this case.

On the other hand, suppose
$(- A_1, - A_2) \in (\mathfrak{D}_1 - [-\eta_0, \eta_0]^2)$, or equivalently
$(A_1, A_2) \in (- \mathfrak{D}_1 + [-\eta_0, \eta_0]^2)$.
Then we have
$$
\left( (- A_1, - A_2)  + [-\eta_0, \eta_0]^2 \right) \cap  \mathfrak{D}_1   \not = \emptyset,
$$
and consequently from (\ref{wshiftbox}) we obtain
\begin{eqnarray}
\label{A1A1inD2}
(- A_1, - A_2)  \in \left( (- A_1, - A_2)  + [-\eta_0, \eta_0]^2 \right) \cup  \mathfrak{D}_1  \subseteq \mathfrak{D}_2 \subseteq \Psi_{\mathbf{v}}(\mathcal{B}_2).
\end{eqnarray}
Therefore, there exists $\mathbf{z}_0 \in \Psi_{\mathbf{v}}^{-1}(\mathfrak{D}_2) \subseteq \mathcal{B}_2$ such that
$\Psi_{\mathbf{v}}(\mathbf{z}_0) = (-A_1, -A_2)$, in other words $\mathbf{z}_0$ is a critical point of $G(\mathbf{u}) + A_1 \log u_1 + A_2 \log u_2$
(see Remark \ref{rem1}).

\section{Estimating the integral - Case (I)}
In this section we will make use of the following lemma, which is a slight variant of \cite[pp. 332, Proposition 2]{St}.
\begin{lem}
\label{firstderivbdd}
Let $\kappa > 0$ be sufficiently small. Suppose $f: \mathbb{R} \rightarrow \mathbb{R}$ is twice differentiable on the interval $(a - \kappa, b + \kappa)$
and $|f'(t)| \geq \mathfrak{c}_1 > 0$ for all $t \in [a,b]$. Let $\psi: \mathbb{R} \rightarrow \mathbb{R}$ be a smooth
function with its support contained in $[a + \kappa, b - \kappa]$. We have the following estimates for two separate cases.

\textnormal{i)} Suppose $|f'' (t)| \leq \mathfrak{c}_2$ for all $t \in [a,b]$.
Then
$$
\Big{|}  \int_a^b \psi(t) e^{i \tau f(t) } dt  \Big{|}  \leq \frac{1}{|\tau|} \left( \frac{2(b-a)}{\mathfrak{c}_1} +  (b - a)^2 \frac{\mathfrak{c}_2}{\mathfrak{c}_1^2} \right) \max_{x \in (a,b) } | \psi'(x) |.
$$

\textnormal{ii)} Suppose $f''$ is continuous on $(a - \kappa, b + \kappa )$ and  $|f''(t)|>0$ on $[a,b]$. Then
$$
\Big{|}  \int_a^b \psi(t) e^{i \tau f(t) } dt  \Big{|}  \leq \frac{1}{|\tau|} \left( \frac{4(b-a)}{\mathfrak{c}_1} \right) \max_{x \in (a,b) } | \psi'(x) |.
$$
\end{lem}
\begin{proof}
First by integration by parts
\begin{eqnarray}
\label{firstineq}
\Big{|} \int_a^b \psi(t) e^{i \tau f(t) } dt \Big{|} &=& \Big{|} \int_a^b  \psi'(x) \int_a^x e^{i \tau f(t) } dt dx \Big{|}
\\
\notag
&\leq& (b-a) \max_{x \in (a,b) } | \psi'(x) | \cdot \max_{a \leq x \leq b}  \Big{|} \int_a^x e^{i \tau f(t) } dt \Big{|}.
\end{eqnarray}
We have
\begin{eqnarray}
\int_a^x e^{i \tau f(t) } dt &=& \int_a^x \frac{1}{i \tau f'(t)} \left(  \frac{d}{dt} e^{i \tau f(t) } \right) dt
\label{secondineq}
\\
\notag
&=&
\frac{1}{i \tau f'(t)}  e^{i \tau f(t) } \Big{|}^x_a - \frac{1}{i \tau}  \int_a^x e^{i \tau f(t) }  \left( \frac{d}{dt}  \frac{1}{f'(t)} \right)  dt.
\end{eqnarray}
For any $a \leq x \leq b$,
$$
\Big{|} \frac{1}{i \tau f'(x)}  e^{i \tau f(x) } - \frac{1}{i \tau f'(a)}  e^{i \tau f(a) } \Big{|} \leq \frac{2}{|\tau| \mathfrak{c}_1}.
$$
Under the assumption of i), we have
\begin{eqnarray}
\label{thirdineq}
\Big{|} - \frac{1}{i \tau} \int_a^x e^{i \tau f(t) }  \left( \frac{d}{dt}  \frac{1}{f'(t)} \right)  dt \Big{|}
=
\Big{|} \frac{1}{i \tau} \int_a^x e^{i \tau f(t) }  \frac{ f''(t)}{ (f'(t))^2}  dt \Big{|} \leq (b - a) \frac{\mathfrak{c}_2}{\mathfrak{c}_1^2 |\tau|}.
\end{eqnarray}
By combining (\ref{firstineq}) and (\ref{secondineq}) with these estimates, we obtain the estimate in i). To obtain the estimate in  ii),
instead of (\ref{thirdineq}) we simply use the following
\begin{eqnarray}
\label{forthineq}
\Big{|} - \frac{1}{i \tau} \int_a^x e^{i \tau f(t) }  \left( \frac{d}{dt}  \frac{1}{f'(t)} \right)  dt \Big{|}
&\leq&
\frac{1}{|\tau|}   \int_a^x  \Big{|} \frac{d}{dt}  \frac{1}{f'(t)}  \Big{|}  dt
\\
&=& \frac{1}{|\tau|}   \Big{|}  \int_a^x  \left( \frac{d}{dt}  \frac{1}{f'(t)}  \right)  dt  \Big{|}
\notag
\\
&=& \frac{1}{|\tau|}   \Big{|} \frac{1}{f'(x)} -  \frac{1}{f'(a)}  \Big{|}
\notag
\\
&\leq& \frac{2}{ \mathfrak{c}_1  |\tau|};
\notag
\end{eqnarray}
we used the fact that $f''$ is continuous and monotone to deduce
$$
\int_a^x  \Big{|} \frac{d}{dt}  \frac{1}{f'(t)}  \Big{|}  dt = \Big{|}  \int_a^x  \left( \frac{d}{dt}  \frac{1}{f'(t)}  \right)  dt  \Big{|}.
$$
\end{proof}

Let $\varpi$ be as in the statement of Proposition \ref{mainprop}, where we choose
$\delta_0$ to satisfy
\begin{eqnarray}
\label{cond on delta0}
\delta_0 < \delta_7
\end{eqnarray}
so that the support of $\varpi$ is contained in the interior of $\mathcal{B}_3 \times \mathcal{B}'_3$ (Recall $\mathcal{B}_3$ and $\mathcal{B}'_3$
were defined in the sentence after (\ref{choicedelta7}).).
We begin by bounding the integral (\ref{mainintegral2}) by
\begin{eqnarray}
\label{mainint1'}
&& \Big{|} \int_{0}^{\infty} \cdots \int_{0}^{\infty} \varpi(\mathbf{x}) x_1^{i t_1} \cdots x_n^{i t_n} e \left( \tau F(\mathbf{x})\right) d\mathbf{x} \Big{|}
\\
\notag
&\ll&
\int_{\mathcal{B}'_3}  \Big{|} \int_{0}^{\infty} \int_{0}^{\infty}  \varpi_{\mathbf{v}}(\mathbf{u}) u_1^{i t_1} u_2^{i t_2} e \left( \tau F(u_1, u_2, \mathbf{v}) \right) d\mathbf{u} \Big{|}
d\mathbf{v},
\end{eqnarray}
where $\varpi_{\mathbf{v}}(\mathbf{u}) = \varpi(\mathbf{u}, \mathbf{v})$.
We now bound the inner integral for each fixed $\mathbf{v} \in \mathcal{B}'_3$.

Recall all the work in Sections \ref{secsecondbox} and \ref{secthirdbox} are under the assumption
$(A_1, A_2) \in [- \lambda_1, \lambda_1] \times [- \lambda_2, \lambda_2]$.
We begin by taking care of the case $(A_1, A_2) \not \in [- \lambda_1, \lambda_1] \times [- \lambda_2, \lambda_2]$.
Without loss of generality suppose $|A_1| > \lambda_1$.
In this case it follows from (\ref{lambdacond2}) that the inequalities
$$
\Big{|} \frac{\partial G}{\partial u_1} (u_1, u_2) + \frac{A_1}{u_1} \Big{|} > \frac{\lambda_1}{2 u_1} \geq \frac{\lambda_1}{2 \rho_{\max}} > 0
$$
and
$$
\Big{|} \frac{\partial^2 G}{\partial u_1^2} (u_1, u_2) - \frac{A_1}{u_1^2} \Big{|} > \frac{\lambda_1}{2 u_1^2} \geq \frac{\lambda_1}{2 \rho_{\max}^2} > 0
$$
hold for all $\mathbf{u} \in \mathcal{B}_3$.
Clearly $(\frac{\partial^2 G}{\partial u_1^2} (u_1, u_2) - \frac{A_1}{u_1^2})$ is continuous on $\mathbf{u} \in \mathcal{B}_3$. Therefore, we obtain from
ii) of Lemma \ref{firstderivbdd} that
\begin{eqnarray}
&&
\label{first bound'2}
\Big{|} \int_{0}^{\infty} \int_{0}^{\infty}  \varpi_{\mathbf{v}}(\mathbf{u}) u_1^{i t_1} u_2^{i t_2} e \left( \tau F(u_1, u_2, \mathbf{v}) \right) d\mathbf{u} \Big{|}
\\
&\leq&
\int_{0}^{\infty} \Big{|}  \int_{0}^{\infty}  \varpi_{\mathbf{v}}(u_1, u_2)  e^{  i 2 \pi \tau \left( G (u_1, u_2)  + A_1 \log u_1 \right) }
d u_1 \Big{|}
d u_2
\notag
\\
\notag
&\ll&
|\tau|^{-1}.
\end{eqnarray}

Now we take care of the case $(A_1, A_2) \in [- \lambda_1, \lambda_1] \times [- \lambda_2, \lambda_2]$.
First suppose $(- A_1, - A_2) \not \in (\mathfrak{D}_1 - [-\eta_0, \eta_0]^2)$, in which case we have that the inequality (\ref{deriv bound1})
holds for all $\mathbf{u} \in \mathcal{B}_3$.
It also follows from (\ref{lambdacond2}) that
$$
\Big{|} \frac{\partial^2 G}{\partial u_{j_0}^2} (u_1, u_2) - \frac{A_{j_0}}{u_{j_0}^2} \Big{|}
\leq \frac{\lambda_{j_0}}{2 \rho_{\min}^2} + \frac{\lambda_{j_0}}{ \rho_{\min}^2 }
$$
holds for all $\mathbf{u} \in \mathcal{B}_3$. If $j_0 = 1$ then we obtain from i) of Lemma \ref{firstderivbdd}
\begin{eqnarray}
&&
\label{first bound'}
\Big{|} \int_{0}^{\infty} \int_{0}^{\infty}  \varpi_{\mathbf{v}}(\mathbf{u}) u_1^{i t_1} u_2^{i t_2} e \left( \tau F(u_1, u_2, \mathbf{v}) \right) d\mathbf{u} \Big{|}
\\
&\leq&
\int_{0}^{\infty} \Big{|}  \int_{0}^{\infty}  \varpi_{\mathbf{v}}(u_1, u_2)  e^{  i 2 \pi \tau \left( G (u_1, u_2)  + A_1 \log u_1 \right) } d u_1 \Big{|}
d u_2
\notag
\\
\notag
&\ll&
|\tau|^{-1}.
\end{eqnarray}
We obtain the same upper bound when $j_0=2$ as well.

Finally, suppose  $(- A_1, - A_2) \in (\mathfrak{D}_1 - [-\eta_0, \eta_0]^2)$, in which case as explained in the sentence after
(\ref{A1A1inD2}) we have that $G(\mathbf{u}) + A_1 \log u_1 + A_2 \log u_2$ has a critical point $\mathbf{z}_0 \in \mathcal{B}_2$.
Then using the notations as in Section \ref{secsecondbox}, we have
\begin{eqnarray}
\label{int2-1'}
&& \Big{|} \int_{0}^{\infty} \int_{0}^{\infty}  \varpi_{\mathbf{v}}(\mathbf{u}) u_1^{i t_1} u_2^{i t_2} e \left( \tau F(u_1, u_2, \mathbf{v}) \right) d\mathbf{u} \Big{|}
\\
&=&
\Big{|} \int_{0}^{\infty}   \int_{0}^{\infty}  \varpi_{\mathbf{v}}(\mathbf{u})  e^{  i 2 \pi \tau \left( G (u_1, u_2)  + A_1 \log u_1 + A_2 \log u_2 \right) } d \mathbf{u} \Big{|}
\notag
\\
&=&
\Big{|} \int_{- \mathbf{z}_0 + \mathcal{B}_3 }  \varpi_{\mathbf{v}}(\mathbf{u} + \mathbf{z}_0)
 e^{  i 2 \pi \tau  \phi (\mathbf{u})  }  d \mathbf{u} \Big{|}.
\notag
\end{eqnarray}
Since $\mathcal{B}_3 \subseteq \mathcal{B}_2$, we know from Section \ref{secsecondbox}
(see the sentence after (\ref{bdd set of B2one})) that
$\mathcal{F}$ is a diffeomorpshism on $(- \mathbf{z}_0 + \mathcal{B}_3)$.
Let $\pi_j: \mathbb{R}^2 \rightarrow \mathbb{R}$ denote the projection on to the $j$-th coordinate.
Then from the definitions of $\phi_{i,j}$, $\mathcal{F}_1$ and $\mathcal{F}_2$, we have
\begin{eqnarray}
\label{eqnsofphi}
\phi(\mathbf{u}) &=& \varepsilon_1 \pi_1^2( \mathcal{F} (\mathbf{u}) ) + \varepsilon_1  \varepsilon_2 \pi^2_2( \mathcal{F} (\mathbf{u}) )
\\
\notag
&=& \varepsilon_1 \pi_1^2( \mathcal{F} (\mathcal{F}^{-1}(\mathbf{y})) )
+ \varepsilon_1 \varepsilon_2 \pi^2_2( \mathcal{F} ( \mathcal{F}^{-1} (\mathbf{y}) ))
\\
\notag
&=& \varepsilon_1 y_1^2 + \varepsilon_1 \varepsilon_2 y_2^2.
\end{eqnarray}
Therefore, we obtain
\begin{eqnarray}
&&  \int_{- \mathbf{z}_0 + \mathcal{B}_3 }  \varpi_{\mathbf{v}}(\mathbf{u} + \mathbf{z}_0)
 e^{  i 2 \pi \tau  \phi (\mathbf{u})  } d \mathbf{u}
\label{int2-2'}
\\
&=& \int_{ \mathcal{F}^{-1}  \mathcal{F}( - \mathbf{z}_0 + \mathcal{B}_3)  }  \varpi_{\mathbf{v}}(\mathbf{u} + \mathbf{z}_0)
 e^{  i 2 \pi \tau  \left( \varepsilon_1 \pi_1^2( \mathcal{F} (\mathbf{u}) ) + \varepsilon_1 \varepsilon_2 \pi^2_2( \mathcal{F} (\mathbf{u}) )   \right)  } d \mathbf{u}
\notag
\\
&=&
\notag
 \int_{\mathcal{F}(   - \mathbf{z}_0 + \mathcal{B}_3   ) }
\varpi_{\mathbf{v}}(\mathcal{F}^{-1}(\mathbf{y}) + \mathbf{z}_0) \
e^{  i 2 \pi \tau  \left( \varepsilon_1 \pi_1^2( \mathcal{F} (\mathcal{F}^{-1} (\mathbf{y}) ) ) + \varepsilon_1 \varepsilon_2 \pi^2_2( \mathcal{F} (\mathcal{F}^{-1} (\mathbf{y}) )   \right)  }
\  | \textnormal{det}( \textnormal{Jac}\mathcal{F}^{-1} (\mathbf{y})) |  \ d \mathbf{y}
\\
&=&
\notag
\int_{\mathcal{F}(   - \mathbf{z}_0 + \mathcal{B}_3   ) }
\varpi_{\mathbf{v}}(\mathcal{F}^{-1}(\mathbf{y}) + \mathbf{z}_0) \  e^{ i 2 \pi \tau  \left( \varepsilon_1 y_1^2 + \varepsilon_1 \varepsilon_2 y_2^2 \right) } \  | \textnormal{det}( \textnormal{Jac}\mathcal{F}^{-1} (\mathbf{y})) |  \ d \mathbf{y}.
\end{eqnarray}

We know $(- \mathbf{z}_0 + \mathcal{B}_3) \subseteq [- 3\delta_6, 3\delta_6]^2$ because of (\ref{bdd set of B2one}).
Then from (\ref{rem1ineq4}) and the definitions of $\mathcal{F}$ and $\phi_{i,j}$,
we can find $L> 0$ such that
$$
\mathcal{F}( - \mathbf{z}_0 + \mathcal{B}_3 ) \subseteq [- L/2, L/2] \times [- L/2, L/2]
$$
holds uniformly over all $\mathbf{z}_0$ and $\mathbf{v}$ in consideration.

Let us define
$$
W(y_1, y_2) =
\left\{
    \begin{array}{ll}
         \varpi_{\mathbf{v}}(\mathcal{F}^{-1}(\mathbf{y}) +  \mathbf{z}_0 ) | \textnormal{det}( \textnormal{Jac}\mathcal{F}^{-1} (\mathbf{y})) |
         &\mbox{  if }  \mathbf{y} \in  \mathcal{F}( - \mathbf{z}_0 + \mathcal{B}_3 ), \\
         0
         &\mbox{  otherwise.} \\
    \end{array}
\right.
$$
\begin{claim}\label{claimW}
$W$ is a smooth function satisfying
\begin{eqnarray}
\label{boundonW''}
\Big{|} \frac{\partial^2 W}{\partial y_1 \partial y_2}(y_1, y_2) \Big{|} \ll 1,
\end{eqnarray}
where the implicit constant is independent of the specific choices of $\mathbf{v} \in \mathcal{B}_3'$, $A_1 \in [- \lambda_1, \lambda_1]$,
$A_2 \in [- \lambda_2, \lambda_2]$, and the critical point $\mathbf{z}_0 \in \mathcal{B}_2$.
\end{claim}
\begin{proof}
First we prove $W$ is smooth. We know that
$\mathcal{F}$ is a diffeomorphism on $(- 4 \delta_6, 4 \delta_6)^2$,
$(- \mathbf{z}_0 + \mathcal{B}_3) \subseteq [- 3 \delta_6, 3 \delta_6]^2$,
and the support of $\varpi_{\mathbf{v}}(\mathcal{F}^{-1}(\mathbf{y}) + \mathbf{z}_0)$ is contained in the interior of
$\mathcal{F}(- \mathbf{z}_0 + \mathcal{B}_3)$.
It is clear that $W$ is smooth on the open set $\mathcal{F}((-4 \delta_6, 4 \delta_6 )^2)$.
For any $\mathbf{u}' \in \mathbb{R}^2 \backslash \mathcal{F}(- \mathbf{z}_0 + \mathcal{B}_3)$ we can find an open set $U$ such that $\mathbf{u}' \in U \subseteq \mathbb{R}^2 \backslash \mathcal{F}(- \mathbf{z}_0 + \mathcal{B}_3)$. Thus we have $W|_{U} \equiv 0$,
and hence smooth on $U$. Since $\mathbb{R}^2 = \mathcal{F}((-4 \delta_6, 4 \delta_6 )^2) \cup (\mathbb{R}^2 \backslash \mathcal{F}(- \mathbf{z}_0 + \mathcal{B}_3))$ we are done.

For the second part of the claim, it suffices to prove the bound (\ref{boundonW''}) for
$\mathbf{y} \in \mathcal{F}((-4 \delta_6, 4 \delta_6 )^2)$ because $W$ is identically $0$ on
$\mathbb{R}^2 \backslash \mathcal{F}(- \mathbf{z}_0 + \mathcal{B}_3)$ as shown above.
Recall $\varpi$ satisfies (\ref{smoothweightbound}).
Since $| \textnormal{det}( \textnormal{Jac}\mathcal{F}^{-1} (\mathbf{y}))| = 1 / |\textnormal{det}( \textnormal{Jac}\mathcal{F}(\mathbf{u}))|$,
where $\mathbf{y} = \mathcal{F}(\mathbf{u})$, and we have (\ref{lowerbounddet}), it follows that $| \textnormal{det}( \textnormal{Jac}\mathcal{F}^{-1} (\mathbf{y}))| \ll 1$ on $\mathbf{y} \in \mathcal{F}((-4 \delta_6, 4 \delta_6 )^2)$.
Let us denote $\mathcal{F}^{-1} = (\mathcal{G}_1, \mathcal{G}_2)$.
In order to achieve (\ref{boundonW''}), by the chain rule we see that all we need now are
upper bounds on (the absolute values of) $\mathcal{G}_1$ and $\mathcal{G}_2$, and also on their first, second and third partial derivatives.
Clearly we have $|\mathcal{G}_j(\mathbf{y})| \leq 4 \delta_6$ $(1 \leq j \leq 2)$ for any $\mathbf{y} \in \mathcal{F}((-4 \delta_6, 4 \delta_6 )^2)$.
In order to obtain upper bounds on the partial derivatives of $\mathcal{G}_1$ and $\mathcal{G}_2$ on $\mathbf{y} \in \mathcal{F}((-4 \delta_6, 4 \delta_6 )^2)$, we use the relation
$$
\textnormal{Jac}\mathcal{F}^{-1} (\mathbf{y}) =  \textnormal{Jac}\mathcal{F}(\mathbf{u})^{-1}
$$
and by the chain rule it suffices to obtain a lower bound for $|\det (\textnormal{Jac}{\mathcal{F}} (\mathbf{u}))|$ on $\mathbf{u} \in (-4 \delta_6, 4 \delta_6)^2$ and also upper bounds for the partial derivatives of $\mathcal{F}_1$ and $\mathcal{F}_2$ on $\mathbf{u} \in (-4 \delta_6, 4 \delta_6)^2$.
This can be done in a similar manner as in Section \ref{secsecondbox} and we omit the details. As a result we obtain
for any $1 \leq i, j, k , \ell \leq 2$, the inequalities
$$
\Big{|} \frac{ \partial  \mathcal{G}_i }{ \partial y_j  } (y_1, y_2) \Big{|} \ll 1, \ \
\Big{|} \frac{ \partial^2  \mathcal{G}_i }{ \partial y_j \partial y_k  } (y_1, y_2) \Big{|} \ll 1 \ \
\textnormal{  and  } \ \
\Big{|} \frac{ \partial^3  \mathcal{G}_i }{ \partial y_j \partial y_k  \partial y_{\ell}} (y_1, y_2) \Big{|} \ll 1
$$
hold for any $\mathbf{y} \in \mathcal{F}((-4 \delta_6, 4 \delta_6 )^2)$, where the implicit constants in these inequalities are
independent of the specific choices of
$\mathbf{v} \in \mathcal{B}_3'$, $A_1 \in [- \lambda_1, \lambda_1]$,
$A_2 \in [- \lambda_2, \lambda_2]$, and the critical point $\mathbf{z}_0 \in \mathcal{B}_2$.
Thus our claim holds.
\end{proof}
Let
$$
\mathcal{I}_1(y_1) = \int_{-L}^{y_1}  e^{ i 2 \pi \tau  \left( \varepsilon_1 t_1^2 \right)}  d t_1 \ \ \textnormal{  and  }  \  \
\mathcal{I}_2(y_2) = \int_{-L}^{y_2}  e^{ i 2 \pi \tau  \left( \varepsilon_1 \varepsilon_2 t_2^2 \right)}  d t_2.
$$
By applying integration by parts twice, the integral in (\ref{int2-2'}) becomes
\begin{eqnarray}
&&\int_{ \mathcal{F}( - \mathbf{z}_0 + \mathcal{B}_3 ) }  W(y_1, y_2) \ e^{ i 2 \pi \tau  \left( \varepsilon_1 y_1^2 + \varepsilon_1 \varepsilon_2 y_2^2 \right) } \ d \mathbf{y}
\label{int2-3'}
\\
\notag
&=&
\int_{-L}^{L}  e^{ i 2 \pi \tau \varepsilon_1 \varepsilon_2 y_2^2  } \int_{-L}^{L}   e^{ i 2 \pi  \tau   \varepsilon_1 y_1^2  } \  W(y_1, y_2) \  dy_1  dy_2
\\
\notag
&=&
- \int_{-L}^{L}  e^{ i 2 \pi \tau \varepsilon_1 \varepsilon_2 y_2^2  }
\int_{-L}^{L} \frac{\partial W}{\partial y_1}(y_1, y_2) \ \mathcal{I}(y_1) \ dy_1  dy_2
\\
&=&
\notag
\int_{-L}^{L} \mathcal{I}(y_1)
\left( \int_{-L}^{L} \frac{\partial^2 W}{\partial y_1 \partial y_2}(y_1, y_2) \  \mathcal{I}_2(y_2) \  dy_2 \right) d y_1.
\end{eqnarray}
In order to bound this integral we use the following.
\begin{claim}
\label{claimint}
Let $- L \leq y_1, y_2 \leq L$. For each $1 \leq j \leq 2$, we have
$$
| \mathcal{I}_j ( y_j ) | \ll \frac{1}{ |\tau|^{1/2} }.
$$
\end{claim}
\begin{proof}
Let $\varepsilon'_1 = \varepsilon_1$ and $\varepsilon'_2 = \varepsilon_1 \varepsilon_2$. Then we have
\begin{eqnarray}
\mathcal{I}_j ( y_j )^2
=
\int_{-L}^{y_j} \int_{-L}^{y_j}  e^{ i 2 \pi \varepsilon'_j \tau \left( s_j^2 + t_j^2 \right)} d s_j d t_j
= \int_0^{2 \pi} \int_{\ell_1(\theta)}^{\ell_2(\theta)}
e^{i 2 \pi \varepsilon'_j \tau  r^2  } r dr d\theta
\end{eqnarray}
for appropriate functions $\ell_1(\theta)$ and $\ell_2(\theta)$ satisfying $0 \leq \ell_1(\theta) \leq \ell_2(\theta) \leq \sqrt{2 L^2}$.
The inner integral becomes
$$
\Big{|} \int_{\ell_1(\theta)}^{\ell_2(\theta)} e^{ i 2 \pi \varepsilon'_j \tau  r^2  } r dr \Big{|}
=
\Big{|} \frac{1}{i 4 \pi \varepsilon'_j \tau }  \left( e^{ i 2 \pi \varepsilon'_j \tau  \ell_2(\theta)^2  }  - e^{ i 2 \pi \varepsilon'_j \tau  \ell_1(\theta)^2  } \right) \Big{|}
\ll
\frac{1}{|\tau|},
$$
and the claim follows immediately.
\end{proof}
It follows from (\ref{int2-1'}), (\ref{int2-2'}), (\ref{int2-3'}), and Claims \ref{claimW} and \ref{claimint} that
\begin{eqnarray}
\label{second bound'}
\Big{|} \int_{0}^{\infty} \int_{0}^{\infty}  \varpi_{\mathbf{v}}(\mathbf{u}) u_1^{i t_1} u_2^{i t_2} e \left( \tau F(u_1, u_2, \mathbf{v}) \right) d\mathbf{u} \Big{|}
&\ll&
\int_{-L}^{L} | \mathcal{I}(y_1) |
\int_{-L}^{L} | \mathcal{I}_2(y_2) |  \ dy_2  d y_1
\\
\notag
&\ll& \frac{1}{|\tau|}.
\end{eqnarray}

Therefore, we obtain from (\ref{mainint1'}), (\ref{first bound'2}), (\ref{first bound'}) and  (\ref{second bound'}) that
$$
\Big{|} \int_{0}^{\infty} \cdots \int_{0}^{\infty} \varpi(\mathbf{x}) x_1^{i t_1} \cdots x_n^{i t_n} e \left( \tau F(\mathbf{x})\right) d\mathbf{x} \Big{|}
\ll \frac{1}{|\tau|},
$$
and this completes the proof of Proposition \ref{mainprop} when $F$ satisfies the hypotheses of Case (I).

\section{Case (II)}
For Case (II) the estimate (\ref{mainintegral2}) can be obtained in a similar manner as in Case (I).
In order to avoid repetition we keep the details to a minimum for this case. Let us suppose $F$ satisfies the hypotheses of Case (II).
Then since $\frac{\partial^2 F}{\partial x_1 \partial x_2} \equiv 0$,
$\mathfrak{G}_{1,2}$ becomes
$$
\mathfrak{G}_{1,2} (\mathbf{x}) = \left( x_1 \frac{\partial^{2} F}{\partial x_1^2 } (\mathbf{x})  + \frac{\partial F}{\partial x_1} (\mathbf{x})   \right) \cdot \left( x_2 \frac{\partial^{2} F}{\partial x_2^2}(\mathbf{x})
+ \frac{\partial F}{\partial x_2}(\mathbf{x})  \right).
$$

\subsection{First Box}
Let $\mathcal{B} = (\mathbf{x}_0 + [- \delta, \delta]^n) \subseteq (0,1)^n$, and we define
\begin{eqnarray}
\label{defm22}
\label{defnm0m12}
m_1 = \min_{1 \leq i \leq 2} \min_{ \mathbf{x} \in  \mathcal{B}} \  \Big{|} x_i \frac{\partial^{2} F}{\partial x_i^2 }(\mathbf{x})  + \frac{\partial F}{\partial x_i} (\mathbf{x}) \Big{|}.
\end{eqnarray}
We let $\rho_{1,1}, \ldots, \rho_{n,1}$, $\rho_{1,2}, \ldots, \rho_{n,2}$, $\rho_{\min}$,
$\rho_{\max}$, $\mathcal{B}_0$ and $\mathcal{B}'_0$ as in (\ref{defnrho1}), (\ref{defnrho2}) and (\ref{defnrho3}).

Let $\mathbf{u} = (u_1, u_2) = (x_1, x_2)$ and $\mathbf{v} = (x_3, \ldots, x_n)$, and also let $\mathbf{u}_0 = (x_{0,1}, x_{0,2})$ and
$\mathbf{v}_0 = (x_{0,3}, \cdots, x_{0,n})$ so that $\mathbf{x}_0 = (\mathbf{u}_0, \mathbf{v}_0)$. Let us define
$\Psi_{\mathbf{v}} = (\Psi_{\mathbf{v}, 1}, \Psi_{\mathbf{v}, 2}): \mathbb{R}^2 \rightarrow \mathbb{R}^2$, where
$$
\Psi_{\mathbf{v},1} (\mathbf{u}) = u_1 \frac{\partial F}{\partial u_1} (u_1, u_2, \mathbf{v})
\ \
\text{  and  }
\ \
\Psi_{\mathbf{v},2} (\mathbf{u}) = u_2 \frac{\partial F}{\partial u_2} (u_1, u_2, \mathbf{v}).
$$
Note
\begin{eqnarray}
\label{detJacpsi22}
\det( \textnormal{Jac}\Psi_{\mathbf{v}}(\mathbf{u})) = \mathfrak{G}_{1,2}(\mathbf{u}, \mathbf{v}).
\end{eqnarray}
\begin{claim}
\label{claim1-2}
Let $\delta_1 > 0$, $\mathcal{B}_1 = (\mathbf{u}_0 + (- \delta_1, \delta_1)^2)$ and  $\mathcal{B}'_1 = (\mathbf{v}_0 + [- \delta_1, \delta_1]^{n-2})$.  Then for $\delta_1 > 0$ sufficiently small, $\Psi_{\mathbf{v}}$ is a diffeomorphism on $\mathcal{B}_1 \subseteq \mathcal{B}_0$ for any $\mathbf{v} \in \mathcal{B}'_1$
\end{claim}
\begin{proof}
We omit the proof as it is similar to that of Claim \ref{claim1}.
\end{proof}
Here we also choose $\delta_1$ to satisfy
\begin{eqnarray}
\label{cond1'onB12}
\delta_1 < \frac{\rho_{\min}}{2}.
\end{eqnarray}
We let
$$
G(\mathbf{u}) = F(u_1, u_2, \mathbf{v}).
$$
Since  $\mathcal{B}_1$ is diffeomorphic to $\Psi_{\mathbf{v}}(\mathcal{B}_1)$, each pair of values
$(\Psi_{\mathbf{v},1}(\mathbf{u}), \Psi_{\mathbf{v},2}(\mathbf{u}))$ gets represented only once over $\mathbf{u} \in \mathcal{B}_1$,
and we also have Remark \ref{rem1}.
Let $\lambda_1$ and $\lambda_2$ be the smallest positive numbers satisfying
\begin{eqnarray}
\label{lambdacond22}
\Big{|} u^2_j \frac{\partial^2 G}{\partial u_j^2} (u_1, u_2)   \Big{|} + \Big{|} u_j \frac{\partial G}{\partial u_j} (u_1, u_2)   \Big{|} \leq \frac{\lambda_j}{2} \ \ (1 \leq j \leq 2)
\end{eqnarray}
for all $\mathbf{u} \in \overline{\mathcal{B}_1}$ and $\mathbf{v} \in {\mathcal{B}'_1}$.

\subsection{Second Box}
Fix $\mathbf{v} \in \mathcal{B}'_1$.
For $\tau \not = 0$, we have
$$
u_1^{i t_1} u_2^{i t_2} e \left( \tau F(u_1, u_2, \mathbf{v}) \right)
= e^{ i 2 \pi \tau \left( G(u_1, u_2) + \frac{t_1}{2 \pi \tau} \log u_1 + \frac{t_2}{2 \pi \tau} \log u_2  \right)}.
$$
Let
$$
A_j = \frac{t_j}{2 \pi \tau} \ \ (1 \leq j \leq 2).
$$
We now deal with the case $A_j \in [- \lambda_j,  \lambda_j]$ $(1 \leq j \leq 2)$.
The critical points of the function $G(\mathbf{u}) + A_1 \log u_1 + A_2 \log u_2$ satisfy
\begin{eqnarray}
\label{crit pt condn2}
\frac{\partial G}{\partial u_j} (u_1, u_2) + \frac{A_j}{u_j} = 0  \ \ \ (1 \leq j \leq 2).
\end{eqnarray}
Suppose $\mathbf{z}_0 = (z_{0,1}, z_{0,2}) \in \mathcal{B}_1$ is a critical point.
Then it follows from our choice of $\mathcal{B}_1$ and $m_1$ (defined in (\ref{defnm0m12})) that
\begin{eqnarray}
\label{m0ineq2}
\Big{|} \left( \frac{\partial^{2} G}{\partial u_1^2 } (\mathbf{z}_0)  - \frac{A_1}{z_{0,1}^2}  \right) \cdot \left( \frac{\partial^{2} G}{\partial u_2^2} (\mathbf{z}_0)
- \frac{A_2}{z_{0,2}^2}
\right) \Big{|}
&=&
\frac{| \mathfrak{G}_{1,2}(\mathbf{z}_0, \mathbf{v} ) | }{ z_{0,1} z_{0,2} }
\\
\notag
&\geq& \frac{m_1^2}{\rho_{\max}^2 }.
\end{eqnarray}

Let $\widetilde{\mathcal{B}}_1 = (\mathbf{u}_0 + [-\delta_1/2, \delta_1/2]^2)$.
Suppose there exists a critical point $\mathbf{z}_0 = (z_{0,1}, z_{0,2}) \in \widetilde{\mathcal{B}}_1$,
in which case we know from Remark \ref{rem1} this is the only critical point in $\widetilde{\mathcal{B}}_1$.
Let us define
\begin{eqnarray}
\phi(\mathbf{u}) &=& G(u_1 + z_{0,1}, u_2 + z_{0,2}) +  A_1 \log (u_1 + z_{0,1}) + A_2 \log (u_2 + z_{0,2})
\\
\notag
&-& G(z_{0,1}, z_{0,2}) - A_1 \log z_{0,1} - A_2 \log z_{0,2}.
\end{eqnarray}
We consider this function over $\mathbf{u} \in (- \mathbf{z}_0 + \widetilde{\mathcal{B}}_1) \subseteq [- \delta_1, \delta_1]^2$.
We have $\phi(\mathbf{0}) = \nabla \phi(\mathbf{0}) = \mathbf{0}$, and $\mathbf{0} \in (- \mathbf{z}_0 + \widetilde{\mathcal{B}}_1)$ is the
only critical point of $\phi(\mathbf{u})$ in $(- \mathbf{z}_0 + \widetilde{\mathcal{B}}_1)$. By the same steps (and the same notations) as before it follows that
$$
\phi(\mathbf{u}) = \sum_{1 \leq j \leq 2} u_j^2 \phi_{j,j}(\mathbf{u}),
$$
where
$$
\phi_{j,j} (\mathbf{u})
=
\sum_{\ell=1}^{d - 2} \frac{H^{(\ell)}_{j,j} (\mathbf{u} ; \mathbf{z}_0 )}{(\ell+1)(\ell+2)} + \frac12 \frac{\partial^2 G}{\partial u_j^2 }(\mathbf{z}_0) -
A_j \sum_{\ell = 2}^{\infty}  \frac{(-1)^{\ell} u_j^{\ell - 2} }{\ell z_{0,j}^{\ell}}.
$$
In particular, it follows from (\ref{defm22}) and (\ref{crit pt condn2}) that
\begin{eqnarray}
\label{m1ineq2}
| \phi_{j,j}(0, 0) | &=&  \Big{|}  \frac12 \cdot \frac{\partial^2 G}{\partial u_j^2}(\mathbf{z}_0) -A_j \frac{1}{2 z_{0,j}^{2}} \Big{|}
\\
\notag
&=& \frac{1}{2 z_{0,j} } \Big{|} z_{0,j} \frac{\partial^2 G}{\partial u_j^2}(\mathbf{z}_0) + \frac{\partial G}{\partial u_j}(\mathbf{z}_0) \Big{|}
\\
\notag
&\geq& \frac{m_1}{2 \rho_{\max}}.
\end{eqnarray}
Also there exists $m_3 > 0$ (depending only on $F$, $\widetilde{\mathcal{B}}_1$ and $\mathcal{B}'_1$) such that
\begin{eqnarray}
\label{m'1ineq2}
| \phi_{j,j}(0, 0) | =  \frac{1}{2 z_{0,j} } \Big{|} z_{0,j} \frac{\partial^2 G}{\partial u_j^2}(\mathbf{z}_0) + \frac{\partial G}{\partial u_j}(\mathbf{z}_0) \Big{|} \leq \frac{  m_3  }{2 \rho_{\min}}   \  \   (1 \leq j \leq 2),
\end{eqnarray}
and therefore
\begin{eqnarray}
\label{m'1ineq'2}
|\phi_{1,1}(\mathbf{0}) \phi_{2,2}(\mathbf{0})| \leq \frac{m_3^2}{ 4 \rho_{\min}^2 }.
\end{eqnarray}
\begin{rem}
\label{rem22}
We have
\begin{eqnarray}
\notag
&& \Big{|} \phi_{1,1}(\mathbf{u}) \phi_{2,2}(\mathbf{u}) -
\frac{1}{4} \left( \frac{\partial^{2} G}{\partial u_1^2 } (\mathbf{z}_0)  - \frac{A_1}{z_{0,1}^2}  \right) \cdot \left( \frac{\partial^{2} G}{\partial u_2^2} (\mathbf{z}_0)
- \frac{A_2}{z_{0,2}^2}
\right) \Big{|}
\label{rem1ineq12}
\\
&\leq&
\notag
C_1  \  \max_{1 \leq j \leq 2} |u_j|  \  \  \  ( - \delta_1/2 \leq u_1, u_2 \leq \delta_1/2),
\end{eqnarray}
and
\begin{eqnarray}
\notag
\Big{|} \phi_{1,1}(\mathbf{u}) -
\frac{1}{2} \left( \frac{\partial^{2} G}{\partial u_1^2 } (\mathbf{z}_0)  - \frac{A_1}{z_{0,1}^2}  \right) \Big{|}
\leq
C_2  \  \max_{1 \leq j \leq 2} |u_j|   \  \  \   (- \delta_1/2 \leq u_1, u_2 \leq \delta_1/2),
\end{eqnarray}
where both $C_1 > 0$ and $C_2 > 0$ are independent of the specific choices of $\mathbf{v} \in \mathcal{B}'_1$,
$A_1 \in [- \lambda_1, \lambda_1]$, $A_2 \in [- \lambda_2, \lambda_2]$, and the critical point $\mathbf{z}_0 \in \widetilde{\mathcal{B}}_1$ (see Remark \ref{rem2}).
\end{rem}
It follows from Remark \ref{rem22} and (\ref{m1ineq2}) that
we can find $\delta_4 > 0$ (in particular satisfying $\delta_4 < \delta_1/2$) and $m_4 > 0$  (both values are independent of the specific choices of $\mathbf{v} \in \mathcal{B}'_1$, $A_1 \in [- \lambda_1, \lambda_1]$, $A_2 \in [- \lambda_2, \lambda_2]$,  and the critical point $\mathbf{z}_0 \in \widetilde{\mathcal{B}}_1$) such that
\begin{eqnarray}
\label{rem1ineq42}
| \phi_{1,1}(\mathbf{u}) | > m_4 \ \ (\mathbf{u} \in [-\delta_4, \delta_4]^2) \ \ \text{ and } \ \ | \phi_{2,2}(\mathbf{u}) | > m_4 \ \ (\mathbf{u} \in [-\delta_4, \delta_4]^2).
\end{eqnarray}

Let $\varepsilon_1$ be the sign of $\phi_{1,1}(\mathbf{u}) $ over $\mathbf{u} \in [-\delta_4, \delta_4]^2$
and let $\varepsilon_2$ be the sign of $\phi_{2,2}(\mathbf{u})$ over $\mathbf{u} \in [-\delta_4, \delta_4]^2$.
We define a new set of variables $\mathbf{y} = (y_1, y_2)$ by
$$
y_1 = \mathcal{F}_1(u_1, u_2)   \  \  \text{  and  }  \  \  y_2 = \mathcal{F}_2(u_1, u_2),
$$
where
$$
\mathcal{F}_1(u_1, u_2) =  u_1 \sqrt{ \varepsilon_1 \phi_{1,1}(\mathbf{u})}
$$
and
$$
\mathcal{F}_2(u_1, u_2) = u_2  \sqrt{ \varepsilon_2 \phi_{2,2}(\mathbf{u})}.
$$
Let $\mathcal{F} = (\mathcal{F}_1, \mathcal{F}_2)$.
It can be verified easily that
$$
\textnormal{Jac}\mathcal{F}(\mathbf{0}) =
\begin{pmatrix}
\sqrt{\varepsilon_1 \phi_{1,1}(\mathbf{0})}  &  0
\\ 0 & \sqrt{\varepsilon_2 \phi_{2,2}(\mathbf{0})}
\end{pmatrix}
$$
and $\det(\textnormal{Jac}\mathcal{F}(\mathbf{0})) \not = 0$.

Let
$$
M =
\frac{m_4}{8}
\left(  \sqrt{  \frac{m_3}{ 2 \rho_{\min} }} \right)^{-1}.
$$
It then follows from (\ref{m'1ineq2}) and (\ref{rem1ineq42}) that
$$
0 < M <  \frac{|\textnormal{Jac}\mathcal{F}(\mathbf{0})|}{ 2 \cdot 2! \cdot   \max_{1 \leq i, j \leq 2}  | [\textnormal{Jac}\mathcal{F}(\mathbf{0})]_{i,j}  |}.
$$

Let $W = (- \delta_5, \delta_5)^2$ with $\delta_5 > 0$ sufficiently small.
In a similar manner as in Section \ref{secsecondbox} we can show $W$ satisfies the three properties of Theorem \ref{thm exp inv} with respect to $\mathcal{F}$, and also that $\delta_5$ can be chosen independently of the specific choices of our parameters.
We also obtain the following.
\begin{claim} There exists $m_5 > 0$ such that the inequality
$$
m_5 \leq \min_{\mathbf{u} \in \partial W} \| \mathcal{F}(\mathbf{u}) - \mathcal{F}(\mathbf{0}) \|
$$
holds uniformly over all choices of $\mathbf{v} \in \mathcal{B}_1'$, $A_1 \in [- \lambda_1, \lambda_1]$,
$A_2 \in [- \lambda_2, \lambda_2]$ (for which $G(\mathbf{u}) + A_1 \log u_1 + A_2 \log u_2$ has a critical point in
$\widetilde{\mathcal{B}}_1$), and the critical point $\mathbf{z}_0 \in \widetilde{\mathcal{B}}_1$.
\end{claim}
\begin{proof}
We omit the proof as it is similar to that of Claim \ref{claim2}.
\end{proof}
We now prove the analogue of Claim \ref{claim3}.
\begin{claim} There exists $\delta_6 > 0$ such that the inequality
$$
\| \mathcal{F}(\mathbf{u}) \| < \frac{m_5}{4}  \ \ (\mathbf{u} \in [- 4 \delta_6, 4 \delta_6]^2)
$$
holds uniformly over all choices of $\mathbf{v} \in \mathcal{B}_1'$, $A_1 \in [- \lambda_1, \lambda_1]$,
$A_2 \in [- \lambda_2, \lambda_2]$ (for which $G(\mathbf{u}) + A_1 \log u_1 + A_2 \log u_2$ has a critical point in
$\widetilde{\mathcal{B}}_1$), and the critical point $\mathbf{z}_0 \in \widetilde{\mathcal{B}}_1$.
\end{claim}
\begin{proof}
We have $\mathcal{F}(\mathbf{0}) = (\mathcal{F}_1(\mathbf{0}), \mathcal{F}_2(\mathbf{0})) = \mathbf{0}$.
It can be verified that the absolute value of each entry of the matrix $\textnormal{Jac}\mathcal{F}(\mathbf{u})$
can be bounded (from above) uniformly over $\mathbf{u} \in [- \delta_4, \delta_4]^2$,  $\mathbf{v} \in \mathcal{B}'_1$, $A_1 \in [\kappa_1, \lambda_1]$,
$A_2 \in [\kappa_2, \lambda_2]$ (for which $G(\mathbf{u}) + A_1 \log u_1 + A_2 \log u_2$ has a critical point in
$\widetilde{\mathcal{B}}_1$), and the critical point $\mathbf{z}_0 \in \widetilde{\mathcal{B}}_1$.
Thus from the mean value theorem it follows that there exists $C_5 >0$ such that
$$
| \mathcal{F}_j(\mathbf{u})  | \leq C_5 \max_{1 \leq j \leq 2} |u_j|  \ \ \  (\mathbf{u} \in  [- \delta_4, \delta_4]^2).
$$
Therefore, there exists $\delta_6 > 0$, depending only on $C_5$, $\delta_4$ and $m_5$, such that
$$
\| \mathcal{F}(\mathbf{u})   \|  < \frac{m_5}{4} \ \ \ (\mathbf{u} \in  [- 4 \delta_6, 4 \delta_6]^2).
$$
\end{proof}

Let us fix $\delta_6 > 0$ sufficiently small, in particular $\delta_6 < \frac15 \min \{  \delta_1, \delta_5  \}$.
Let
$$
V = \{ \mathbf{y} \in \mathbb{R}^2 :  \| \mathbf{y} \| < m_5 /2  \}.
$$
Then we have $[- 3\delta_6, 3 \delta_6]^2 \subseteq \mathcal{F}^{-1}(V)$; in fact, by Theorem \ref{thm exp inv} $\mathcal{F}$ is a diffeomorphism from $(- 4 \delta_6, 4 \delta_6)^2$ to an open subset of $V$. Let $\mathcal{B}_2 = (\mathbf{u}_0 + [-\delta_6, \delta_6]^2)$. We have
$\mathcal{B}_2 \subseteq \widetilde{\mathcal{B}}_1$, and
given any $\mathbf{z} \in \mathcal{B}_2$,
\begin{eqnarray}
\label{delta72}
\mathcal{B}_2 \subseteq (\mathbf{z} + [-3\delta_6, 3\delta_6]^2) \subseteq \mathcal{B}_1
\end{eqnarray}
and
$$
(- \mathbf{z}\ + \mathcal{B}_2) \subseteq [-3\delta_6, 3\delta_6]^2.
$$
Therefore, it follows that given any $\mathbf{v} \in \mathcal{B}_1'$, $A_1 \in [- \lambda_1, \lambda_1]$ and $A_2 \in [- \lambda_2, \lambda_2]$,
and if $G(\mathbf{u}) + A_1 \log u_1 + A_2 \log u_2$ has a critical point $\mathbf{z}_0 \in \mathcal{B}_2$, then
$\mathcal{F} |_{- \mathbf{z}_0 + \mathcal{B}_2}$ is a diffeomorphism.

After this point we can carry out essentially exactly as in Case (I)
to obtain (\ref{mainintegral2}) for Case (II), and we omit the remaining details; one of the differences is that instead of (\ref{eqnsofphi}) we have
\begin{eqnarray}
\phi(\mathbf{u}) &=& \varepsilon_1 \pi_1^2( \mathcal{F} (\mathbf{u}) ) + \varepsilon_2 \pi^2_2( \mathcal{F} (\mathbf{u}) )
\\
\notag
&=& \varepsilon_1 y_1^2 +  \varepsilon_2 y_2^2.
\end{eqnarray}
This concludes the proof of Proposition \ref{mainprop}.
\section{Deduction of Theorem \ref{mainthm}}
\label{secdeducemain}

Before we can deduce Theorem \ref{mainthm} from Proposition \ref{mainprop}, we need to collect few results.

\begin{lem}
\label{irredlemma}
Let $G \in \mathbb{C}[x_1, \ldots, x_n]$ be a homogeneous form of degree $d$.
If $(n - \dim V_G^*) > 2$, then $G$ is irreducible over $\mathbb{C}$.
\end{lem}
\begin{proof}
We prove the statement by contradiction. Suppose $G$ factors over $\mathbb{C}$, say $G(\mathbf{x}) = H(\mathbf{x}) \cdot K(\mathbf{x})$ where
$H$ and $K$ are non-constant homogeneous forms.
By taking partial derivatives we have
$$
\frac{\partial G}{\partial x_j} = K \frac{\partial H }{\partial x_j}  + H \frac{\partial K }{\partial x_j} \ \ (1 \leq j \leq n).
$$
Thus it follows that
$$
V(H) \cap V(K) \subseteq V_{G}^*,
$$
and hence
$$
n-2 \leq  \dim (V(H) \cap V(K))  \leq \dim V_{G}^*.
$$
The above inequality is equivalent to $(n - \dim V_{G}^*) \leq 2$, which is a contradiction. Therefore, $G$ is irreducible over $\mathbb{C}$.
\end{proof}
\begin{lem}
\label{Lemma on the B rank}
Let $G \in \mathbb{C}[x_1, \ldots, x_n]$ be a homogeneous form of degree $d > 1$.
Let $Q \in \mathbb{C}[x_2, \ldots, x_n]$ be defined by $Q(x_2, \ldots, x_n) = G(0, x_2, \ldots, x_n)$.
Then
$$
(n-1) - \dim V_{Q}^* \geq n - \dim V_{G}^* - 2.
$$
Here $\dim V_{Q}^*$ is the dimension of $V_{Q}^*$
as an affine variety in $\mathbb{A}_{\mathbb{C}}^{n-1}$.
\end{lem}
\begin{proof}
We refer the reader to \cite[Lemma 3.1]{SS} for the proof  (The minor oversight in \cite[Lemma 3]{CM} is corrected here.).
\end{proof}

For $1 \leq i < j \leq n$, let
$$
\mathfrak{G}_{i,j}(\mathbf{x}) = \left( x_i \frac{\partial^{2} F}{\partial x_i^2 } (\mathbf{x})  + \frac{\partial F}{\partial x_i} (\mathbf{x})   \right) \cdot \left( x_j \frac{\partial^{2} F}{\partial x_j^2}(\mathbf{x})
+ \frac{\partial F}{\partial x_j}(\mathbf{x})  \right) -
x_i x_j \left( \frac{\partial^{2} F}{\partial{x_i} \partial{x_j}}(\mathbf{x}) \right)^2.
$$
\begin{prop}
\label{prop8.3}
Let $F \in \mathbb{C}[x_1, \ldots, x_n]$ be a homogeneous form of degree $d > 1$ satisfying $(n - \dim V_{F}^*) > 4$.
Then there exist $1 \leq i < j \leq n$ such that $F$ does not divide $\mathfrak{G}_{i,j}$.
\end{prop}
\begin{proof}
We prove the statement in two steps: Step 1. There exist $1 \leq i < j \leq n$ such that
$\frac{\partial^{2} F}{\partial x_i \partial x_j } \equiv 0$; Step 2. For all $1 \leq i < j \leq n$ we have
$\frac{\partial^{2} F}{\partial x_i \partial x_j }  \not \equiv 0$. 

Without loss of generality we assume $F(\mathbf{x}) \not = F(\mathbf{x})|_{x_j =0}$ for any $1 \leq j \leq n$,
and this implies
$$
x_j \frac{\partial^{2} F}{\partial x_j^2 }  + \frac{\partial F}{\partial x_j} \not \equiv 0 \ \ (1 \leq j \leq n).
$$
For each $1 \leq j \leq n$, let us denote
$$
x_j \frac{\partial^{2} F}{\partial x_j^2 }(\mathbf{x})  + \frac{\partial F}{\partial x_j}(\mathbf{x}) = x_1^{\nu_j} K_j (\mathbf{x}),
$$
where
$\nu_j \in \mathbb{Z}_{\geq 0}$ and
$K_j$ is a homogeneous form not divisible by $x_1$. Clearly, $K_j$ has degree less than or equal to
$(d-1)$.

Step 1. Without loss of generality let us assume $i=1$ and $j=2$. Suppose $F$ divides $\mathfrak{G}_{1,2}$;
there exists a homogeneous form $P \in \mathbb{C}[x_1, \ldots, x_n]$ such that
$$
x_1^{\nu_1 + \nu_2} K_1 (\mathbf{x}) K_2 (\mathbf{x}) = F(\mathbf{x}) P(\mathbf{x}).
$$
Since $F$ is irreducible over $\mathbb{C}$ (Lemma \ref{irredlemma}), it follows that it must divide one of
$x_1$, $K_1$ or $K_2$ but this is not possible as $\deg F = d > 1$; therefore, $F$ does not divide $\mathfrak{G}_{1,2}$.

Step 2. First we prove the existence of $1 \leq i' \not = j' \leq n$ such that
there exists a monomial of $F(\mathbf{x})$, with a non-zero coefficient, which is divisible by $x_{i'}$ but not divisible by $x_{j'}$.
Let us suppose otherwise, in which case it follows that
$(x_1 \cdots x_n) | F$ because every $x_i$ appears in at least one of
the monomials of $F(\mathbf{x})$ with a non-zero coefficient. However, this is not possible as $F$ is irreducible over $\mathbb{C}$. Thus there exist $1 \leq i' \not = j' \leq n$ such that there exists a monomial of $F(\mathbf{x})$, with a non-zero coefficient, which is divisible by $x_{i'}$ but not divisible by $x_{j'}$. Without loss of generality let us assume $i'=2$ and $j'=1$.
For simplicity we denote $\mathbf{y} = (x_3, \ldots, x_n)$.
We decompose $F(\mathbf{x})$ in the following manner
$$
F(x_1, x_2, \mathbf{y} ) = F_1(x_1, \mathbf{y}) + G(x_1, x_2, \mathbf{y})
+ F_2(x_2, \mathbf{y}) + F(0, 0, \mathbf{y} ),
$$
where $F_1 (x_1,\mathbf{y} )$ is the sum of all monomials of $F(\mathbf{x})$ which are divisible by $x_1$
but not by $x_2$, $F_2 (x_2,\mathbf{y} )$ is the sum of all monomials of $F(\mathbf{x})$ which are divisible by $x_2$
but not by $x_1$, and $G(x_1, x_2, \mathbf{y})$ is the sum of all monomials of $F(\mathbf{x})$ which are divisible by both $x_1$ and $x_2$. From our assumption ($i'=2$ and $j'=1$) we know that
$F_2(x_2,\mathbf{y} ) \not \equiv 0$ and $x_2 | F_2 (x_2,\mathbf{y} )$. Also since we are in Step 2 we have
$G \not \equiv 0$.

Given a non-zero polynomial $g \in \mathbb{C}[x_1, \ldots x_n] \backslash \{ 0\}$  we let $\nu_{x_1}(g) = \nu \geq 0$ be the number such that
$x_1^{\nu} | g$ but $x_1^{\nu + 1} \nmid g$.
Next we prove
$$
\nu_1 = \nu_{x_1} \left( x_1 \frac{\partial^{2} F}{\partial x_1^2 }  + \frac{\partial F}{\partial x_1}  \right)
=
\nu_{x_1} \left( \frac{\partial F}{\partial x_1}  \right)
$$
and
\begin{eqnarray}
\label{nu2}
\nu_2 = \nu_{x_1} \left( x_2 \frac{\partial^{2} F}{\partial x_2^2 }  + \frac{\partial F}{\partial x_2}  \right)
=
\nu_{x_1} \left( \frac{\partial F}{\partial x_2} \right) = 0.
\end{eqnarray}
Let us denote
$$
F(\mathbf{x}) = H_{d} x_1^{d} + H_{d-1} x_1^{d-1} + \cdots + H_{\ell} x_1^{\ell} + H_0,
$$
where each $H_j$ is either the zero polynomial or a homogeneous form of degree $(d-j)$ in $x_2$ and $\mathbf{y}$, and in particular $\ell \geq 1$ and $H_{\ell} \not \equiv 0$.
Suppose $\ell > 1$. Clearly $(\ell (\ell - 1) H_{\ell}  + \ell H_{\ell})\not \equiv 0$; therefore, we have
\begin{eqnarray}
\label{L-1equality}
\nu_{x_1} \left(  x_1 \frac{\partial^{2} F}{\partial x_1^2 }  + \frac{\partial F}{\partial x_1} \right)
= \ell - 1
= \nu_{x_1} \left(  \frac{\partial F}{\partial x_1} \right).
\end{eqnarray}
On the other hand, suppose $\ell = 1$. In this case, since $\ell H_{\ell} \not \equiv 0$ we obtain (\ref{L-1equality}) as well.

For the second equality (\ref{nu2}), first we note $\frac{\partial F_2}{\partial x_2} \not \equiv 0$.
From
\begin{eqnarray}
\label{8.3}
\frac{\partial F}{\partial x_2} (x_1, x_2, \mathbf{y}) = \frac{\partial G}{\partial x_2} (x_1, x_2, \mathbf{y})
+ \frac{\partial F_2}{\partial x_2} (x_2, \mathbf{y}),
\end{eqnarray}
it is easy to deduce
\begin{eqnarray}
\nu_{x_1} \left(  \frac{\partial F}{\partial x_2} \right) = 0
\end{eqnarray}
because $x_1$ divides $ \frac{\partial G}{\partial x_2} (x_1, x_2, \mathbf{y})$ but it does not divide $\frac{\partial F_2}{\partial x_2} (x_2, \mathbf{y})$.
By the definition of $G(x_1, x_2, \mathbf{y})$ it is easy to see that $\frac{\partial G}{\partial x_2} (x_1, x_2, \mathbf{y})$ and $x_2 \frac{\partial^2  G}{\partial x_2^2} (x_1, x_2, \mathbf{y})$ are both homogeneous forms divisible by $x_1$ (the latter possibly being the zero polynomial).
From (\ref{8.3}) we have
\begin{eqnarray}
x_2 \frac{\partial^2 F}{\partial x_2^2}  (x_1, x_2, \mathbf{y})= x_2 \frac{\partial^2 G}{\partial x_2^2} (x_1, x_2, \mathbf{y})
+ x_2 \frac{\partial^2 F_2}{\partial x_2^2}  (x_2, \mathbf{y}).
\end{eqnarray}
It can be easily verified that
$$
x_2 \frac{\partial^2 F_2}{\partial x_2^2} + \frac{\partial F_2}{\partial x_2} \not \equiv 0,
$$
and since it does not involve any $x_1$ variables, in particular it is not divisible by $x_1$.
Therefore, we obtain
$$
\nu_{x_1} \left( x_2 \frac{\partial^{2} F}{\partial x_2^2 }  + \frac{\partial F}{\partial x_2}  \right)
=
\nu_{x_1} \left( x_2 \frac{\partial^{2} F_2}{\partial x_2^2 }  + \frac{\partial F_2}{\partial x_2}  \right)
= 0,
$$
and we have established the second equality (\ref{nu2}).

We have
\begin{eqnarray}
\frac{\partial F}{\partial x_1} = \frac{\partial F_1}{\partial x_1} + \frac{\partial G}{\partial x_1}.
\end{eqnarray}
Since every monomial of $\frac{\partial G}{\partial x_1}$ is divisible by $x_2$
while none of the monomials of $\frac{\partial F_1}{\partial x_1}$ is divisible by $x_2$, it follows that
$$
\nu_1 = \nu_{x_1} \left( \frac{\partial F}{\partial x_1}  \right)
= \min \{
\nu_{x_1} \left( \frac{\partial F_1}{\partial x_1}  \right),
\nu_{x_1} \left( \frac{\partial G}{\partial x_1}  \right)
\}.
$$
Furthermore, we have
$$
\frac{\partial^2 F }{\partial x_1 \partial x_2} = \frac{\partial^2 G }{\partial x_1 \partial x_2} \not \equiv 0.
$$
Let
$$
\mu = \nu_{x_1} \left( \frac{\partial^2 G }{\partial x_1 \partial x_2}  \right)
\ \ \textnormal{  and  }  \  \
\frac{\partial^2 F }{\partial x_1 \partial x_2} (\mathbf{x}) = x_1^{\mu} K_3(\mathbf{x}).
$$
It turns out
\begin{eqnarray}
\label{eqnmu}
\mu = \nu_{x_1} \left( \frac{\partial G}{\partial x_1} \right).
\end{eqnarray}
To see this let us denote
$$
G(x_1, x_2, \mathbf{y}) = \sum_{ (i_1, i_2) \in I } B_{i_1, i_2}(\mathbf{y}) x_1^{i_1} x_2^{i_2},
$$
where $I \subseteq (\mathbb{N} \times \mathbb{N})$ is the subset of $(\mathbb{N} \times \mathbb{N})$ such that
each $B_{i_1, i_2}$ is a non-zero polynomial in $\mathbf{y}$ for all $(i_1, i_2) \in I$. It is then clear that
$$
\nu_{x_1} \left( \frac{\partial G}{\partial x_1}  \right) = \min_{ \substack{ (i_1, i_2) \in I  } } (i_1 - 1) = \nu_{x_1} \left( \frac{\partial^2 G }{\partial x_1 \partial x_2}  \right) =  \mu.
$$
Therefore, we have
\begin{eqnarray}
\label{eqnnu1}
\nu_1 = \min \{
\nu_{x_1} \left( \frac{\partial F_1}{\partial x_1}  \right),
\nu_{x_1} \left( \frac{\partial G}{\partial x_1}  \right)
\} \leq \mu <  2 \mu + 1.
\end{eqnarray}
Now suppose $F$ divides $\mathfrak{G}_{1,2}$. Recall $F$ is irreducible over $\mathbb{C}$. Then from (\ref{nu2}) and  (\ref{eqnnu1})
it follows that there exists a homogeneous form $P \in \mathbb{C}[x_1, \ldots, x_n]$ such that
$$
x_1^{\nu_1} K_1(\mathbf{x}) K_2(\mathbf{x}) - x_1^{2 \mu + 1} x_2 K^2_3(\mathbf{x}) = F(\mathbf{x}) x_1^{\nu_1} P(\mathbf{x}).
$$
Consequently, we obtain
$$
K_1(\mathbf{x}) K_2(\mathbf{x}) - x_1^{2 \mu + 1 - \nu_1} x_2 K^2_3(\mathbf{x}) = F(\mathbf{x}) P(\mathbf{x}).
$$
By setting $x_1 = 0$ and noting (\ref{eqnnu1}), we obtain
$$
K_1(0, x_2, \mathbf{y}) K_2(0, x_2, \mathbf{y}) = F(0, x_2, \mathbf{y}) P(0, x_2, \mathbf{y}).
$$
Since $K_1$ and $K_2$ are not divisible by $x_1$, we have $K_1(0, x_2, \mathbf{y}) \not \equiv 0$ and $K_2(0, x_2, \mathbf{y}) \not \equiv 0$; therefore,
it follows that $F(0, x_2, \mathbf{y}) \not \equiv 0$ and $ P(0, x_2, \mathbf{y}) \not \equiv 0$ as well.
Our assumption $(n - \dim V_{F}^*) > 4$ and Lemmas  \ref{irredlemma} and \ref{Lemma on the B rank} imply $F(0, x_2, \mathbf{y})$ is irreducible over $\mathbb{C}$. Thus $F(0, x_2, \mathbf{y})$ must divide one of $K_1(0, x_2, \mathbf{y})$ or $K_2(0, x_2, \mathbf{y})$,
but this is not possible as $\deg F(0, x_2, \mathbf{y}) = d$ while $\deg K_1(0, x_2, \mathbf{y}) < d$ and $\deg K_2(0, x_2, \mathbf{y}) < d$; we have obtained a contradiction. Therefore, $F$ does not divide $\mathfrak{G}_{1,2}$.
\end{proof}

Finally, we have the following proposition which we prove using algebraic geometry over $\mathbb{R}$ in Appendix \ref{alggeomreal}.
\begin{prop}
\label{PropApp}
Let $F \in \mathbb{R}[x_1, \ldots, x_n]$ be a homogeneous form of degree $d$ irreducible over $\mathbb{R}$.
Let $H \in \mathbb{R}[x_1, \ldots, x_n]$ be a non-zero polynomial such that $F \nmid H$.
Suppose there is a non-singular point of $V(F; \mathbb{R})$ inside $(0,1)^n$.
Then there exists an open (with respect to the Euclidean topology) set $U \subseteq (0,1)^n$ such that
there exists a non-singular point of $V(F; \mathbb{R})$  inside $U$
and $U \cap V(H; \mathbb{R}) = \emptyset.$
\end{prop}

We are now in position to deduce Theorem \ref{mainthm}. Let $F$ be as in the statement of the theorem.
In particular, by Lemma \ref{irredlemma} it follows that $F$ is irreducible over $\mathbb{C}$.
Without loss of generality we suppose $F(\mathbf{x}) \not = F(\mathbf{x})|_{x_j =0}$ for any $1 \leq j \leq n$ so that
$$
x_j \frac{\partial^{2} F}{\partial x_j^2 }  + \frac{\partial F}{\partial x_j} \not \equiv 0 \ \ (1 \leq j \leq n).
$$
Also without loss of generality let us suppose $F \nmid \mathfrak{G}_{1,2}$ (Proposition \ref{prop8.3}).
We consider the following two cases

Case (I'):  $\frac{\partial^{2} F}{\partial x_1 \partial x_2 }  \not \equiv 0$;

Case (II'): $\frac{\partial^{2} F}{\partial x_1 \partial x_2} \equiv 0$.
\newline
\newline
Clearly we are in one of the two cases.

Let us begin with Case (I'). In this case we know $F$ does not divide any one of
$$
\frac{\partial^{2} F}{\partial x_1 \partial x_2}, \ \ x_1 \frac{\partial^{2} F}{\partial x_1^2}
+ \frac{\partial F}{\partial x_1} \ \ \textnormal{  and  } \ \   x_2 \frac{\partial^{2} F}{\partial x_2^2}
+ \frac{\partial F}{\partial x_2},
$$
because they are all non-zero polynomials of degrees less than or equal to $(d-1)$ while $\deg F = d$.
Let
$$
H_1 = \mathfrak{G}_{1,2} \cdot \frac{\partial^{2} F}{\partial x_1 \partial x_2} \cdot \left( x_1 \frac{\partial^{2} F}{\partial x_1^2}
+ \frac{\partial F}{\partial x_1}
\right)
\cdot \left(  x_2 \frac{\partial^{2} F}{\partial x_2^2}
+ \frac{\partial F}{\partial x_2} \right).
$$
Then it follows that $F \nmid H_1$. Therefore, we obtain from
Proposition \ref{PropApp} that $F$ satisfies the hypotheses of Case (I).
Similarly, for Case (II') we know $F$ does not divide either one of
$$
x_1 \frac{\partial^{2} F}{\partial x_1^2}
+ \frac{\partial F}{\partial x_1} \ \ \textnormal{  and  } \ \   x_2 \frac{\partial^{2} F}{\partial x_2^2}
+ \frac{\partial F}{\partial x_2},
$$
because they are both non-zero polynomials of degrees $(d-1)$ while $\deg F = d$.
Let
$$
H_2 = \mathfrak{G}_{1,2} \cdot
 \left( x_1 \frac{\partial^{2} F}{\partial x_1^2}
+ \frac{\partial F}{\partial x_1}
\right)
\cdot \left(  x_2 \frac{\partial^{2} F}{\partial x_2^2}
+ \frac{\partial F}{\partial x_2} \right).
$$
Then it follows that $F \nmid H_2$. Therefore, we obtain from
Proposition \ref{PropApp} that $F$ satisfies the hypotheses of Case (II).
Thus we have obtained that if $F$ satisfies the hypotheses of Theorem \ref{mainthm}, then
it follows that $F$ satisfies the hypotheses of Proposition \ref{PropApp}; therefore, a special case of Theorem \ref{mainthm}, where $\mathbf{r} = (r_1, \ldots, r_n) = \mathbf{0}$, follows from Proposition \ref{mainprop}.

In order to achieve the result with $\mathbf{r} \in  [\theta_1, \theta_1'] \times \cdots \times [\theta_n, \theta_n'] $ as stated in Theorem \ref{mainthm},
we  consider the smooth weight to be
$$
\varpi(\mathbf{x}) x_1^{r_1} \cdots x_n^{r_n}
$$
in the case of Theorem \ref{mainthm} we have obtained, and the result follows.
\appendix

\section{Algebraic Geometry over the reals}
\label{alggeomreal}
The goal of this appendix is to prove Proposition \ref{PropApp}.
First we recall some terminologies from algebraic geometry over the reals. The main source of the material in
this appendix is \cite{BCR}. Since $\mathbb{R}$ is not algebraically closed,
Hilbert Nullstellensatz does not hold in this setting; for example, a non-constant polynomial $x_1^2 + 1 \in \mathbb{R}[x_1]$ does not have a zero in $\mathbb{R}$.

We denote by $\mathbb{A}^n_{\mathbb{R}}$ the space $\mathbb{R}^n$ with the Zariski topology, i.e.
the closed sets of $\mathbb{A}^n_{\mathbb{R}}$ are of the form
$$
\{ \mathbf{z} \in \mathbb{R}^n:  f(\mathbf{x}) = 0 \ (f \in J)  \}
$$
where $J \subseteq \mathbb{R}[x_1, \ldots, x_n]$; these closed sets are referred to as affine varieties (not necessarily irreducible).
Since $\mathbb{R}[x_1, \ldots, x_n]$ is a Noetherian ring, it follows that that every closed set in $\mathbb{A}^n_{\mathbb{R}}$
is in fact a zero locus of finitely many polynomials in $\mathbb{R}[x_1, \ldots, x_n]$.
Given a subset $S \subseteq \mathbb{R}^n$ we let
$$
I(S) = \{ f \in \mathbb{R}[x_1, \ldots, x_n] : f(\mathbf{x}) = 0 \ (\mathbf{x} \in S)  \},
$$
which is an ideal of $\mathbb{R}[x_1, \ldots, x_n]$.
Let $W \subseteq \mathbb{A}^n_{\mathbb{R}}$ be an affine variety. We define $\dim_{\mathbb{A}^n_{\mathbb{R}}} W$ to be
the Krull dimension of the ring $\mathbb{R}[x_1, \ldots, x_n] / I(W)$, i.e. the largest number $k$ such that
there exist prime ideals $\mathcal{P}_i \subseteq \mathbb{R}[x_1, \ldots, x_n] / I(W)$ $(0 \leq i \leq k)$ such that
$$
\mathcal{P}_0 \subsetneq \mathcal{P}_1 \subsetneq \cdots \subsetneq \mathcal{P}_k.
$$
The number of irreducible components of $W$ is finite, and $\dim_{\mathbb{A}^n_{\mathbb{R}}} W$
is in fact the maximum of the dimensions of its irreducible components. We will not consider the empty set as
being irreducible.

For any affine variety $W \subseteq \mathbb{A}^n_{\mathbb{R}}$ we have $V( I(W); \mathbb{R} ) = W$.
In particular, if $W$ and $X$ are affine varieties such that $W \subsetneq X$ then $I(X) \subsetneq I(W)$.
Also $W$ is irreducible if and only if $I(W)$ is a prime ideal \cite[Theorem 2.8.3 (ii)]{BCR}.
Thus it follows that if $W$ and $X$ are both irreducible affine varieties and $W \subsetneq X$, then
\begin{eqnarray}
\label{dimreln}
\dim_{\mathbb{A}^n_{\mathbb{R}}} W = n - \textnormal{height } I(W)  < n - \textnormal{height } I(X)  =  \dim_{\mathbb{A}^n_{\mathbb{R}} } X.
\end{eqnarray}
Let $I(W) = (P_1, \ldots, P_k)$. For $W$ irreducible we say $\mathbf{z} \in W$ is a non-singular point of $W$ if
the rank of the matrix $[\frac{\partial P_i}{\partial x_j} (\mathbf{z})]$ is equal to $n - \dim_{\mathbb{A}^n_{\mathbb{R}} } (W)$.
Also affine varieties in $\mathbb{A}^n_{\mathbb{R}}$ are closed subsets of $\mathbb{R}^n$ with respect to the
Euclidean topology.

\begin{rem}
\label{remA1}
Here we defined non-singular points of an irreducible $W$ in terms of $I(W)$.
Let $F \in \mathbb{R}[x_1, \ldots, x_n]$ be a non-constant homogenous form irreducible over $\mathbb{R}$
(In this case, the Krull dimension of $\mathbb{R}[x_1, \ldots, x_n] / (F)$ is $(n-1)$.),
and suppose there exists $\mathbf{z} \in \mathbb{R}^n$ and $1 \leq j_0 \leq n$ such that  $\partial F/ \partial x_{j_0} (\mathbf{z}) \not = 0$.
Then it can be verified that $(F) = I (V(F; \mathbb{R}))$ (see \cite[Definition 3.3.1 and Proposition 3.3.16]{BCR}). Therefore, under these assumptions on $F$ the definition of non-singularity for points in $V(F;\mathbb{R})$ given in Section \ref{secintro} agrees with that defined above.
\end{rem}

\begin{lem}
\label{A1}
Suppose $\emptyset \not = V \subseteq \mathbb{A}_{\mathbb{R}}^n$ is irreducible. Let $U = (a_1, b_1) \times \cdots \times (a_n, b_n)$ be an open (with respect to the Euclidean topology) set where every point of $V \cap U$ is non-singular. Then $\overline{(V \cap U)}^{Zar} = V$, where $\overline{(V \cap U)}^{Zar}$ is the Zariski closure of $(V \cap U)$ in $\mathbb{A}_{\mathbb{R}}^n$.
\end{lem}
\begin{proof}
Let us denote
$$
X = \overline{(V \cap U)}^{Zar} \subseteq V.
$$
If $X = V$ then we are done, so let us suppose $X \subsetneq V$. Clearly this implies
$(X \cap U) \subseteq (V \cap U)$. From the definition of $X$ it follows easily that
$(V \cap U) \subseteq (X \cap U)$. Therefore, we have $(V \cap U) = (X \cap U)$.
Let $\dim_{\mathbb{A}^n_{\mathbb{R}} } V = k$. Then it follows from (\ref{dimreln}) that $m = \dim_{\mathbb{A}^n_{\mathbb{R}} } X < k$.

Suppose there exists a non-singular point $\mathbf{z} \in X$ of dimension $m$
in the sense of \cite[Definition 3.3.9]{BCR} (when $X$ is irreducible this is equivalent to the non-singularity defined above for points on irreducible varieties) such that $\mathbf{z} \in (X \cap U)$. By \cite[Proposition 3.3.10]{BCR} there exists an irreducible component $W_0$ of $X$ such that it is the unique irreducible component of $X$
containing $\mathbf{z}$ and it is a non-singular (as defined above for irreducible affine varieties) point of $W_0$. Thus by taking a sufficiently small open box $U_1$ such that $\mathbf{z} \in U_1 \subseteq U$, we have
$$
(V \cap U_1) = (W_0 \cap U_1).
$$
In particular, $\mathbf{z}$ is also a non-singular point of $V$ because of our hypothesis on $(V \cap U)$.
It then follows from \cite[Proposition 3.3.10]{BCR} and the implicit function theorem that for sufficiently small open box $U_2$ such that $\mathbf{z} \in U_2 \subseteq U_1$,
we have $(V \cap U_2)$ is a $k$-dimensional manifold while  $(W_0 \cap U_2)$ is an $m$-dimensional manifold.
However, since $(V \cap U_2) = (W_0 \cap U_2)$ and $k \not = m$, this is a contradiction by the invariance of domain theorem.

Now suppose there is no non-singular point of $X$ of dimension $m$ contained in $X \cap U$.
Let $\textnormal{Sing}(X) \subseteq \mathbb{A}^n_{\mathbb{R}}$ denote the set of points in $X$ that are not non-singular points of $X$ of dimension $m$;
it follows from \cite[Proposition 3.3.14]{BCR} that $\textnormal{Sing}(X)$ is an affine variety
satisfying $\textnormal{Sing}(X) \subsetneq X$.
In this case $(X \cap U) \subseteq \textnormal{Sing}(X)  \subsetneq X$, and by taking the Zariski closure we obtain (from the first paragraph of this proof)
$$
X = \overline{(V \cap U)}^{Zar}  = \overline{(X \cap U)}^{Zar} \subseteq \textnormal{Sing}(X)  \subsetneq X;
$$
this is a contradiction. Therefore, we can not have $X \subsetneq V$, i.e. we have obtained $X = V$.
\end{proof}

\begin{proof}[Proof of Proposition \ref{PropApp}]
First from Remark \ref{remA1} we have $(F) = I(V(F;\mathbb{R}))$, and since $(F) \subseteq \mathbb{R}[x_1, \ldots, x_n]$ is prime it follows that $V(F;\mathbb{R})$ is irreducible.
Suppose every non-singular point of $V(F; \mathbb{R})$ in $(0,1)^n$ is contained in
$V(H; \mathbb{R})$. Let $\mathbf{z}$ be a non-singular point of $V(F; \mathbb{R})$ in $(0,1)^n$ which we know exists.
Let $U$ be an open box with sufficiently small side lengths satisfying $\mathbf{z} \in U \subseteq (0,1)^n$ and every point in $V(F;\mathbb{R}) \cap U$ is a non-singular point of $V(F;\mathbb{R})$. It follows that $V(F;\mathbb{R}) \cap U \subseteq V(H;\mathbb{R})$.
Then from Lemma \ref{A1} and the definition of the Zariski closure, we have
\begin{eqnarray}
\label{dimreln2}
V(F;\mathbb{R})  =  \overline{U \cap V(F;\mathbb{R})}^{Zar} \subseteq V(F;\mathbb{R}) \cap V(H;\mathbb{R}).
\end{eqnarray}
This implies $V(F;\mathbb{R}) \subseteq V(H;\mathbb{R})$. Then we have
\begin{eqnarray}
\notag
(F) = I (V(F;\mathbb{R})) \supseteq  I (V(H;\mathbb{R})) \supseteq (H).
\end{eqnarray}
However, this is not possible as $F \nmid H$; we have obtained a contradiction. Therefore, it is not possible that every non-singular point
of $V(F; \mathbb{R})$ in $(0,1)^n$ is contained in $V(H; \mathbb{R})$.

Let $\mathbf{z}$ be a non-singular point of $V(F; \mathbb{R})$ in $(0,1)^n$ which is not contained in $V(H; \mathbb{R})$.
Then there exists an open set $U \subseteq (0,1)^n$ containing $\mathbf{z}$ such that
$U \cap V(H; \mathbb{R}) = \emptyset$, because $ V(H; \mathbb{R})$ is closed
with respect to the Euclidean topology. 
\end{proof}

\section{Explicit inverse function theorem}
\label{AppB}
In this appendix we establish Theorem \ref{thm exp inv} by following a proof of the inverse function theorem; for this we used the proof of \cite[Theorem 2-11]{Sp} as the main reference.

We let $\| \cdot \|$ be the $L^2$-norm on $\mathbb{R}^n$. We begin by stating the following basic lemma from linear algebra.
We leave the details to the reader.
\begin{lem}
\label{lin alg lemma}
Let $B$ be an $n \times n$ matrix with real entries, and
let $b_{\max}$ be the maximum of the absolute values of the entries of $B$.
Then
$$
\|B \mathbf{x} \| \leq b_{\max} n \| \mathbf{x}  \|.
$$
\end{lem}

Let $\mathfrak{F} = (\mathfrak{F}_1, \ldots, \mathfrak{F}_n): \mathbb{R}^n \rightarrow \mathbb{R}^n$ and suppose $A = \textnormal{Jac} \mathfrak{F} (\mathbf{x}_0)$ is invertible.
Let $a_{\max}$ be the maximum of the absolute values of the entries of $A$.
Let $W \subseteq \mathbb{R}^n$ be a bounded convex open set such that

i) $\mathbf{x}_0 \in W$,

ii) $\textnormal{det}\left( \textnormal{Jac}\mathfrak{F}(\mathbf{x}) \right) \not = 0$ $(\mathbf{x} \in W)$, and

iii)
$$
\Big{|} \frac{\partial \mathfrak{F}_i}{ \partial x_j} (\mathbf{x}) - \frac{\partial \mathfrak{F}_i}{ \partial x_j} (\mathbf{x}_0) \Big{|} < M \ \ \ (\mathbf{x} \in W, \ 1 \leq i, j \leq n),
$$
where $0 < M < |\det A| / (n \cdot n! \cdot a_{\max}^{n-1})$.

\begin{claim}
\label{claimB2}
Given any $\mathbf{x}_1, \mathbf{x}_2 \in W$, we have
\begin{eqnarray}
\label{ineq app1}
\| \mathbf{x}_1 - \mathbf{x}_2 \| < \frac{ n!  a_{\max}^{n-1}}{ |\det A|  - n M n! a_{\max}^{n-1} } \| \mathfrak{F}(\mathbf{x}_1) - \mathfrak{F}(\mathbf{x}_2) \|.
\end{eqnarray}
\end{claim}
\begin{proof}
By iii) for any $\mathbf{x} \in W$ we have
$$
\Big{|}  \frac{\partial}{\partial x_j}  [\mathfrak{F}(\mathbf{x}) - A \mathbf{x} ]_i   \Big{|}
=
\Big{|}  \frac{\partial \mathfrak{F}_i}{\partial x_j} (\mathbf{x}) - \frac{ \partial \mathfrak{F}_i}{\partial x_j}(\mathbf{x}_0)   \Big{|} < M,
$$
where $[\mathfrak{F}(\mathbf{x}) - A \mathbf{x} ]_i$ denotes the $i$-th coordinate of $(\mathfrak{F}(\mathbf{x}) - A \mathbf{x})$.
Then by the triangle inequality and the mean value theorem (and the fact that $W$ is convex), we have
\begin{eqnarray}
\| A \mathbf{x}_1 - A \mathbf{x}_2 \| - \|\mathfrak{F}(\mathbf{x}_1) - \mathfrak{F}(\mathbf{x}_2) \|
\notag
&\leq& \|  (\mathfrak{F}(\mathbf{x}_1) - \mathfrak{F}(\mathbf{x}_2)) - (A \mathbf{x}_1 - A \mathbf{x}_2)  \|
\\
\notag
&=& \| (\mathfrak{F}(\mathbf{x}_1) - A \mathbf{x}_1  ) - (\mathfrak{F}(\mathbf{x}_2) - A \mathbf{x}_2  ) \|
\\
\notag
&\leq&
n M \| \mathbf{x}_1 - \mathbf{x}_2 \|.
\end{eqnarray}
It follows from Lemma \ref{lin alg lemma} that
$$
\| \mathbf{x}_1 - \mathbf{x}_2 \|  =  \| A^{-1} (A \mathbf{x}_1) - A^{-1} (A \mathbf{x}_2) \|
\leq   \frac{n! a_{\max}^{n-1}}{|\det A|}    \| A \mathbf{x}_1 - A \mathbf{x}_2 \|.
$$
Here we used $(n-1)! a_{\max}^{n-1}$ as an upper bound for the maximum value of the absolute values of the entries of $\textnormal{adj}A$, the adjugate of $A$. Therefore, we obtain
\begin{eqnarray}
\frac{|\det A| }{ n! a_{\max}^{n-1} } \| \mathbf{x}_1 - \mathbf{x}_2 \| - \|\mathfrak{F}(\mathbf{x}_1) - \mathfrak{F}(\mathbf{x}_2) \|
\leq
n M \| \mathbf{x}_1 - \mathbf{x}_2 \|
\end{eqnarray}
from which the statement follows immediately.
\end{proof}

Let $\partial W = \overline{W} \backslash W$ where $\overline{W}$ is the closure of $W$.
In particular, $\partial W$ is compact.
Let
$$
m = \min_{\mathbf{x} \in \partial W} \| \mathfrak{F}(\mathbf{x}) - \mathfrak{F}(\mathbf{x}_0) \|,
$$
which exists because the function $\| \mathfrak{F}(\mathbf{x}) - \mathfrak{F}(\mathbf{x}_0) \|$ is continuous.

We define
$$
V = \{ \mathbf{y} \in \mathbb{R}^n : \| \mathbf{y} - \mathfrak{F}(\mathbf{x}_0) \| < m/2  \}.
$$
Since $ \| \mathbf{y} - \mathfrak{F}(\mathbf{x}_0) \|$ is continuous, $V$ is open.
\begin{lem}
\label{lemma 1 app}
Given $\mathbf{y} \in V$ there exists a unique $\mathbf{x} \in W$ such that $\mathfrak{F}(\mathbf{x}) = \mathbf{y}$.
\end{lem}
\begin{proof}
Let us fix $\mathbf{y} \in V$. Consider $h : W \rightarrow \mathbb{R}$ defined by
$$
h(\mathbf{x}) = \| \mathbf{y} - \mathfrak{F}(\mathbf{x}) \|^2 = \sum_{i=1}^n
(y_i - \mathfrak{F}_i(\mathbf{x}))^2.
$$
Since $h$ is continuous it attains a minimum value on $\overline{W}$.
Let $\mathbf{z} \in \overline{W}$ be a point at which the minimum is attained. In fact it must be that
$\mathbf{z} \in W$, because for any $\mathbf{x} \in \partial W$ we have
$$
\sqrt{h(\mathbf{x}_0)} = \| \mathbf{y} - \mathfrak{F}(\mathbf{x}_0) \| < m/2 < \| \mathfrak{F}(\mathbf{x}) - \mathfrak{F}(\mathbf{x}_0) \| -\| \mathbf{y} - \mathfrak{F}(\mathbf{x}_0) \|   \leq \| \mathbf{y} - \mathfrak{F}(\mathbf{x}) \| = \sqrt{h(\mathbf{x})}.
$$
Then since $h$ is differentiable we have $\nabla h (\mathbf{z}) = \mathbf{0}$, and this is equivalent to
$$
0 = \frac{\partial h}{\partial x_j} (\mathbf{z}) = \sum_{i=1}^n 2  (y_i - \mathfrak{F}_i (\mathbf{z}))
\ \frac{\partial \mathfrak{F}_i}{ \partial x_j} (\mathbf{z}) \ \ \  (1 \leq j \leq n).
$$
Therefore, we have
$$
\mathbf{0} = \textnormal{Jac} \mathfrak{F} (\mathbf{z})^T \cdot ( \mathbf{y} - \mathfrak{F}(\mathbf{z})),
$$
where $\textnormal{Jac} \mathfrak{F} (\mathbf{z})^T$ is the transpose of $\textnormal{Jac} \mathfrak{F} (\mathbf{z})$,
and by ii) it follows that
$$
\mathbf{y} - \mathfrak{F}(\mathbf{z}) = \mathbf{0}.
$$
We now prove this point is unique. Suppose $\mathbf{y} = \mathfrak{F}(\mathbf{x}_1) = \mathfrak{F}(\mathbf{x}_2)$ for $\mathbf{x}_1, \mathbf{x}_2 \in W$.
Then it follows from (\ref{ineq app1}) that
$$
\| \mathbf{x}_1 - \mathbf{x}_2 \| \leq \frac{ n!  a_{\max}^{n-1}}{ |\det A|  - n M n! a_{\max}^{n-1} } \| \mathfrak{F}(\mathbf{x}_1) - \mathfrak{F}(\mathbf{x}_2) \| = 0.
$$
\end{proof}

By Lemma \ref{lemma 1 app} we see that the inverse of $\mathfrak{F}$ is well-defined on $V$.
Let $U = \mathfrak{F}^{-1}(V) \subseteq W$. Since $\mathfrak{F}$ is continuous and $V$ is open, it follows that $U$ is open.
\begin{lem}
\label{lemma 2 app}
$\mathfrak{F}^{-1}|_{V}$ is continuous and differentiable.
\end{lem}
\begin{proof}
Given any $\mathbf{y}_i \in V$ there exists a unique $\mathbf{x}_i \in U$ such that $\mathfrak{F}(\mathbf{x}_i) = \mathbf{y}_i$ $(1 \leq i \leq 2)$.
By (\ref{ineq app1}) we have
\begin{eqnarray}
\| \mathfrak{F}^{-1}(\mathbf{y}_1) - \mathfrak{F}^{-1}(\mathbf{y}_2)   \|  &=& \| \mathbf{x}_1 - \mathbf{x}_2 \|
\label{B.2'}
\\
\notag
&\leq&
\frac{ n!  a_{\max}^{n-1}}{ |\det A|  - n M n! a_{\max}^{n-1} }
\| \mathfrak{F}(\mathbf{x}_1) - \mathfrak{F}(\mathbf{x}_2)   \|
\\
\notag
&=& \frac{ n!  a_{\max}^{n-1}}{ |\det A|  - n M n! a_{\max}^{n-1} } \| \mathbf{y}_1 - \mathbf{y}_2 \|.
\end{eqnarray}
Therefore, it follows that $\mathfrak{F}^{-1}|_{V}$ is continuous.

Fix $\mathbf{y}_3 \in V$ and let $\mathbf{y}_3 = \mathfrak{F}(\mathbf{x}_3)$, where $\mathbf{x}_3 \in U$.
For simplicity let us denote $B = \textnormal{Jac} \mathfrak{F} (\mathbf{x}_3)$, and define
\begin{eqnarray}
\label{eqnofphi}
\phi(\mathbf{h}) = \mathfrak{F} (\mathbf{x}_3 + \mathbf{h}) - \mathfrak{F} (\mathbf{x}_3) - B \mathbf{h}.
\end{eqnarray}
It follows from the differentiability of $\mathfrak{F}(\mathbf{x})$ at $\mathbf{x}_3$ that
\begin{eqnarray}
\label{limit phi}
\lim_{\|\mathbf{h}\| \rightarrow 0} \frac{ \|\phi(\mathbf{h}) \|  }{\|\mathbf{h} \| } = 0.
\end{eqnarray}
Let $\mathbf{h} \not = \mathbf{0}$ be such that $(\mathbf{x}_3 + \mathbf{h}) \in U$ and denote $\mathbf{y}_3' = \mathfrak{F}(\mathbf{x}_3 + \mathbf{h})$.
By rearranging the equation (\ref{eqnofphi}) we obtain
$$
- B^{-1} \phi(\mathbf{h}) = (\mathbf{x}_3 + \mathbf{h}) - \mathbf{x}_3 -
B^{-1} (\mathfrak{F}(\mathbf{x}_3 + \mathbf{h}) - \mathfrak{F}(\mathbf{x}_3) ),
$$
and this is equivalent to
$$
- B^{-1} ( \phi(\mathfrak{F}^{-1}(\mathbf{y}'_3) -  \mathfrak{F}^{-1}(\mathbf{y}_3)))
=
\mathfrak{F}^{-1}(\mathbf{y}'_3) -  \mathfrak{F}^{-1}(\mathbf{y}_3) - B^{-1}( \mathbf{y}'_3 - \mathbf{y}_3).
$$
Therefore, we have
\begin{eqnarray}
\frac{ \|\mathfrak{F}^{-1}(\mathbf{y}'_3) -  \mathfrak{F}^{-1}(\mathbf{y}_3) - B^{-1}( \mathbf{y}'_3 - \mathbf{y}_3) \|  }{\|\mathbf{y}'_3 - \mathbf{y}_3 \| }
\label{eqn 1' app}
&=&
\frac{ \|B^{-1} \  \phi(\mathfrak{F}^{-1}(\mathbf{y}'_3) -  \mathfrak{F}^{-1}(\mathbf{y}_3)) \|  }{\|\mathbf{y}'_3 - \mathbf{y}_3 \| }
\\
&\leq&
\frac{ C' \|  \phi(\mathfrak{F}^{-1}(\mathbf{y}'_3) -  \mathfrak{F}^{-1}(\mathbf{y}_3)) \|  }{\|\mathbf{y}'_3 - \mathbf{y}_3 \| }
\notag
\end{eqnarray}
for some $C' > 0$ dependent on $B^{-1}$; the last inequality was obtained by Lemma \ref{lin alg lemma}.
Using (\ref{B.2'}) we also have
\begin{eqnarray}
\frac{ \|  \phi(  \mathfrak{F}^{-1}(\mathbf{y}'_3) -  \mathfrak{F}^{-1}(\mathbf{y}_3)  ) \|  }{\|\mathbf{y}'_3 - \mathbf{y}_3 \| }
\label{B.3''}
&=&
\frac{ \| ( \phi(\mathfrak{F}^{-1}(\mathbf{y}'_3) -  \mathfrak{F}^{-1}(\mathbf{y}_3))) \|  }{ \| \mathfrak{F}^{-1}(\mathbf{y}'_3) -  \mathfrak{F}^{-1}(\mathbf{y}_3) \| } \cdot \frac{\| \mathfrak{F}^{-1}(\mathbf{y}'_3) -  \mathfrak{F}^{-1}(\mathbf{y}_3) \|}{\|\mathbf{y}'_3 - \mathbf{y}_3 \| }
\\
&\leq&
\frac{ \| ( \phi(\mathfrak{F}^{-1}(\mathbf{y}'_3) -  \mathfrak{F}^{-1}(\mathbf{y}_3))) \|  }{ \| \mathfrak{F}^{-1}(\mathbf{y}'_3) -  \mathfrak{F}^{-1}(\mathbf{y}_3) \| } \cdot \frac{ n!  a_{\max}^{n-1}}{ |\det A|  - n M n! a_{\max}^{n-1} }.
\notag
\end{eqnarray}
Then by (\ref{limit phi}) and the continuity of $\mathfrak{F}^{-1}$, we have
\begin{eqnarray}
\label{B.3'''}
\lim_{\|\mathbf{y}'_3 - \mathbf{y}_3\| \rightarrow 0}
\frac{ \| ( \phi(\mathfrak{F}^{-1}(\mathbf{y}_3') -  \mathfrak{F}^{-1}(\mathbf{y}_3))) \|  }
{ \| \mathfrak{F}^{-1}(\mathbf{y}_3') -  \mathfrak{F}^{-1}(\mathbf{y}_3) \| } = 0.
\end{eqnarray}
Thus letting $\mathbf{k} = \mathbf{y}'_3 - \mathbf{y}_3$, it follows from (\ref{eqn 1' app}), (\ref{B.3''}) and (\ref{B.3'''}) that
$$
\lim_{\|\mathbf{k}\| \rightarrow 0} \frac{ \|\mathfrak{F}^{-1}(\mathbf{y}_3 + \mathbf{k}) -  \mathfrak{F}^{-1}(\mathbf{y}_3) - B^{-1} \mathbf{k} \|  }{\|\mathbf{k} \| } =0,
$$
and hence $\mathfrak{F}^{-1}$ is differentiable at $\mathbf{y}_3$.
\end{proof}

Let $\mathcal{X}$ be an open subset of $\mathbb{R}^n$.
Let $\mathcal{C}^k(\mathcal{X})$ denote the set of functions from $\mathcal{X}$ to $\mathbb{R}^n$ which are $k$-times differentiable and all of
their $k$-th partial derivatives are continuous on $\mathcal{X}$. We now prove by induction that if $\mathfrak{F}(\mathbf{x}) \in \mathcal{C}^k(U)$, then
$\mathfrak{F}^{-1}(\mathbf{y}) \in \mathcal{C}^k(V)$ for all $k \in \mathbb{N}$; from this it follows that
if $\mathfrak{F}$ is smooth on $U$ then $\mathfrak{F}^{-1}$ is smooth on $V$.
The base case is proved in the proof of Lemma \ref{lemma 2 app}.
The proof also shows that given $\mathfrak{F}(\mathbf{x}_0) = \mathbf{y}_0 \in V$, we have
$$
\textnormal{Jac} \mathfrak{F}^{-1}(\mathbf{y}_0) = (\textnormal{Jac} \mathfrak{F}(\mathbf{x}_0))^{-1}
=
(\textnormal{Jac} \mathfrak{F} (    \mathfrak{F}^{-1}( \mathbf{y}_0)  )  )^{-1}.
$$
Suppose the statement holds for some $k \in \mathbb{N}$. If $\mathfrak{F} (\mathbf{x}) \in \mathcal{C}^{k+1}(U)$ then
the $(i,j)$-th entry of $\textnormal{Jac} \mathfrak{F} (  \mathfrak{F}^{-1}( \mathbf{y})  )$,
$$
\frac{\partial \mathfrak{F}_i}{\partial x_j}   (    \mathfrak{F}^{-1}( \mathbf{y})  ),
$$
is a composition of two $k$-times continuously differentiable functions. Since each entry of the inverse of a matrix is a smooth function (on the open set where the determinant is non-zero) in its entries, we see that each entry of $(\textnormal{Jac} \mathfrak{F} (    \mathfrak{F}^{-1}( \mathbf{y})  )  )^{-1} = \textnormal{Jac} \mathfrak{F}^{-1}(\mathbf{y})$ is a $k$-times continuously differentiable function; therefore, it follows that $\mathfrak{F}^{-1}(\mathbf{y}) \in \mathcal{C}^{k+1}(V)$.

\end{document}